\documentclass{lms} 
\title{A categorification of Grassmannian cluster algebras}
\date{10 June 2016}

\author{Bernt Tore Jensen, Alastair King and Xiuping Su}

\classno{13F60 (primary), 14M15, 16D90 (secondary)}

\usepackage{tikz}
\usetikzlibrary{arrows}

\usepackage{ifthen}
\usepackage{color}
\usepackage[normalem]{ulem}

\newcommand{\New}[1]{#1}

\usepackage{amssymb,amsmath}
\newtheorem{theorem}{Theorem}[section]
\newtheorem{proposition}[theorem]{Proposition}
\newtheorem{lemma}[theorem]{Lemma}
\newtheorem{corollary}[theorem]{Corollary}
\newtheorem{observation}[theorem]{Observation}

\newnumbered{definition}[theorem]{Definition}
\newnumbered{example}[theorem]{Example}
\newnumbered{remark}[theorem]{Remark}

\newcommand{\sm}{\small} 

\renewcommand{\setminus}{\smallsetminus}
\renewcommand{\epsilon}{\varepsilon}
\renewcommand{\emptyset}{\varnothing}
\renewcommand{\leq}{\leqslant}
\renewcommand{\geq}{\geqslant}

\newcommand{\sfrac}[2]{{\textstyle\frac{#1}{#2}}}
\newcommand{\lra}[1]{\xrightarrow{#1}}
\newcommand{\isom}{\cong}
\newcommand{\sub}{\subseteq}
\newcommand{\from}{\leftarrow}
\newcommand{\eps}{\varepsilon}

\newcommand{\aka}{\emph{a.k.a.}}

\newcommand{\CC}{\mathbb{C}}
\newcommand{\NN}{\mathbb{N}}
\newcommand{\QQ}{\mathbb{Q}}
\newcommand{\ZZ}{\mathbb{Z}}
\newcommand{\N}{\mathcal{N}}
\newcommand{\Atil}{\widetilde{A}}

\newcommand{\Mtil}{\widetilde{M}}
\newcommand{\Ntil}{\widetilde{N}}

\newcommand{\GL}{\operatorname{GL}}
\newcommand{\SL}{\operatorname{SL}}
\newcommand{\CM}{\operatorname{CM}}

\newcommand{\md}{\operatorname{mod}}
\newcommand{\Sub}{\operatorname{Sub}}
\newcommand{\Sym}{\operatorname{Sym}}
\newcommand{\Hom}{\operatorname{Hom}}
\newcommand{\SHom}{\operatorname{\underline{Hom}}}
\newcommand{\End}{\operatorname{End}}
\newcommand{\Ext}{\operatorname{Ext}}
\newcommand{\add}{\operatorname{add}}
\newcommand{\Aut}{\operatorname{Aut}}

\newcommand{\rk}{\operatorname{rk}}
\newcommand{\len}{\operatorname{len}}
\newcommand{\soc}{\operatorname{soc}}
\newcommand{\img}{\operatorname{im}}

\renewcommand{\k}{\mathbb{C}} 

\newcommand{\rootlat}{\operatorname{\Lambda}}
\newcommand{\gr}{\operatorname{\delta}}
\newcommand{\op}{^{op}}

\newcommand{\Pluck}[1]{\Phi_{#1}}
\newcommand{\cycint}[1]{[#1]}
\newcommand{\grass}[2]{G_{#1,#2}}
\newcommand{\GCA}[2]{\CC\left[\grass{#1}{#2}\right]}
\newcommand{\tgr}{*}
\newcommand{\dual}{^\vee}
\newcommand{\trun}{\pi}
\newcommand{\Grot}{K}

\newcommand{\maxNC}{\mathcal{S}}
\newcommand{\cluschar}[1]{\varphi_{#1}}
\newcommand{\cluschartil}[1]{\widetilde{\varphi}_{#1}}

\newcommand{\homcc}[1]{\psi_{#1}}
\newcommand{\homog}[2]{\operatorname{hmg}({#1};{#2})}
\newcommand{\dehomog}{\nu}
\newcommand{\irrep}[1]{V_{#1}}

\newcommand{\intvl}[2]{\left\{#1,\ldots ,#2\right\}}

\newcommand{\thfrac}[3]{
\begin{aligned} #1\\ \hline \\[-3.3\jot] #2 \\ \hline \\[-3.3\jot] #3 \end{aligned}
}

\newcommand{\Ab}{\overline{A}}
\newcommand{\Ah}{\widehat{A}}
\newcommand{\Rb}{\overline{R}}
\newcommand{\Rh}{\widehat{R}}
\newcommand{\Zb}{\overline{Z}}
\newcommand{\Zh}{\widehat{Z}}

\begin{document} 
\maketitle

\begin{abstract}
We describe a ring whose category of Cohen-Macaulay modules 
provides an additive categorification \New{of}
the cluster algebra structure on the homogeneous coordinate ring of the 
Grassmannian of $k$-planes in $n$-space.
More precisely, there is a cluster character defined on the category 
which maps the rigid indecomposable objects to the cluster variables
and the maximal rigid objects to clusters. 
This is proved by showing that the quotient of this category by a single 
projective-injective object is Geiss-Leclerc-Schr\"oer's category Sub~$Q_k$, 
which categorifies the coordinate ring of the big cell in this Grassmannian.
\end{abstract}

\section{Introduction} 

Let $\GCA{k}{n}$ denote the homogeneous coordinate ring
of the Grassmannian of $k$-dimensional quotient spaces of $\k^n$.
As a representation of $\GL_n(\k)$, we know (e.g. by the Borel-Weil Theorem) that
\begin{equation}\label{eq:repsum}
  \GCA{k}{n} = \bigoplus_{d\geq 0} \irrep{d\omega},
\end{equation}
where the representation $\irrep{\omega}=\Lambda^k(\k^n)$ generates $\GCA{k}{n}$ as an algebra 
and $\irrep{d\omega}$ is the irreducible summand of $\Sym^d \irrep{\omega}$ of highest weight $d\omega$.
Under the action of the diagonal torus of $\GL_n(\k)$, an eigenbasis of $\Lambda^k(\k^n)$ 
consists of Pl\"ucker coordinates $\Pluck{I}$, for each $k$-subset $I$ of $\intvl{1}{n}$.
Thus $\GCA{k}{n}$ is identified with a quotient of the algebra of polynomials in the $\Pluck{I}$
by an ideal generated by certain well-known quadratic relations, the Pl\"ucker relations
(see e.g. \cite{HP}).
The shortest such Pl\"ucker relations may be written in the form
\begin{equation}\label{eq:pluckrel}
   \Pluck{Jac}\Pluck{Jbd} = \Pluck{Jab}\Pluck{Jcd} + \Pluck{Jad}\Pluck{Jbc},
\end{equation}
where $J$ is any $(k-2)$-subset of $\intvl{1}{n}$ disjoint from $\{a,b,c,d\}$
and $Jxy$ denotes $J\cup\{x,y\}$.

As a notable initial example in their newly developed theory of cluster algebras,
Fomin-Zelevinsky \cite[\S 12.2]{FZ} showed that these short Pl\"ucker relations
may be considered as the exchange relations for a cluster algebra structure on $\GCA{2}{n}$,
for which the cluster variables are precisely the Pl\"ucker coordinates
and the clusters are in bijection with the triangulations of an $n$-gon.
Scott~\cite{Sc06} then showed that this cluster algebra structure can be generalised to $\GCA{k}{n}$,
but with the addition of cluster variables of higher degree (and more exchange relations).
In all cases, these cluster algebras have $n$ `frozen' variables
$\Pluck{\cycint{j}}$, for $j=1,\ldots,n$,
where 
\begin{equation}\label{eq:[j]}
\cycint{j}=\intvl{j+1}{j+k}
\end{equation}
is a cyclic interval, i.e. addition is mod $n$.

Amongst the Grassmannian cluster algebras $\GCA{k}{n}$, the ones with finite cluster type,
i.e. with finitely many cluster variables, 
are those with $k=2$ or $n-2$, of cluster type $A_{n-3}$, and
those with $k=3$ or $n-3$ and $n=6,7,8$, 
of cluster types $D_4$, $E_6$, $E_8$, respectively.

This numerology has a parallel in the numerology of simple singularities:
the plane curve singularity $x^k=y^{n-k}$ is a simple singularity in precisely
these cases and its singularity type is the cluster type \New{(see \cite[Ch.10]{Yo})}. 
This paper will go some way to showing that this is not a coincidence.

Let $R$ be the (complete) coordinate ring of the singularity $x^k=y^{n-k}$ 
and $G\leq \SL_2(\k)$ be the cyclic group of order $n$ that acts naturally on it 
(see Section~\ref{sec:cat} for precise definitions).
The main aim of the paper is to show that the category $\CM_G(R)$
of $G$-equivariant (maximal) Cohen-Macaulay $R$-modules is an `additive categorification'
of the cluster algebra $\GCA{k}{n}$. 
In particular, the \New{reachable} rigid indecomposable modules in $\CM_G(R)$
correspond \New{one-to-one} to the cluster variables.
In fact, we will mostly work in the language of the non-commutative algebra 
 $A=R\tgr G$, the twisted group ring, and the category $\CM(A)$ of 
 Cohen-Macaulay $A$-modules, which coincides with $\CM_G(R)$.
 
The result is proved by relating it closely to the celebrated 
and more general result of Geiss-Leclerc-Schr\"oer \cite{GLS08},
which, in our context, proves the corresponding relationship between the cluster algebra structure on the coordinate ring
\begin{equation}\label{eq:kN}
\k[\N] = \GCA{k}{n} / (\Pluck{\cycint{n}} - 1)
\end{equation}
of the affine open cell in the Grassmannian and a subcategory 
$\Sub Q_k$ of the module category $\md \Pi(A_{n-1})$ of the preprojective algebra of type $A_{n-1}$
(see Remark~\ref{rem:subQk} for precise definitions).
Indeed, as \eqref{eq:kN} might suggest, our main result (Theorem~\ref{thm:main}) proves that
$\Sub Q_k$ is (equivalent to) a quotient of $\CM(A)$ by the 
projective-injective object $P_n$ corresponding to $\Pluck{\cycint{n}}$. 

As observed in \cite{GLS08}, for any $N$ in $\Sub Q_k$,
the cluster character $\cluschar{N}\in \k[\N]$ may be `homogenised'
to $\cluschartil{N}\in\GCA{k}{n}$ with degree equal to $\dim\soc(N)$,
the dimension of the socle of $N$.
The map $N\mapsto \cluschar{N} \mapsto \cluschartil{N}$ provides a one-to-one correspondence between
the \New{reachable} rigid indecomposables in $\Sub Q_k$, the cluster variables in $\k[\N]$ 
and the cluster variables except $\Pluck{\cycint{n}}$ in $\GCA{k}{n}$.
A key part of our result is that  $\dim\soc(N)$ is equal to the rank of the `minimal' lift of $N$ into $\CM(A)$.

After proving our main result, 
we make (in Section~\ref{sec:rkone}) a more careful study of rank one modules in $\CM(A)$,
showing that they correspond precisely to the $k$-subsets of $\intvl{1}{n}$,
i.e. to the Pl\"ucker coordinates. We also show (Corollary~\ref{cor:genfilt})
that every rigid indecomposable has a `generic' filtration by rank one modules,
which enables us to describe it in terms of its `profile', given by the $k$-subsets
that correspond to the rank one modules in this filtration.
In particular, this profile determines the class of the module in the Grothendieck
group of $\CM(A)$, which we describe in more detail in Section~\ref{sec:grotgrp}.
In Section~\ref{sec:examples}, we give the Auslander-Reiten quivers of $\CM(A)$ 
in the finite cluster-type cases $k=3$, $n=6,7,8$.

We conclude, in Section~\ref{sec:conc}, by showing how to define a 
homogeneous cluster character $\homcc{M}\in\GCA{k}{n}$ for $M$ in $\CM(A)$, 
so that $\deg\homcc{M}=\rk M$ and $\homcc{P_n}=\Pluck{\cycint{n}}$.
This induces a direct bijection between
the \New{reachable} rigid indecomposables in $\CM(A)$ and
the cluster variables in $\GCA{k}{n}$ (Theorem~\ref{thm:clusvar}).

\emph{Acknowledgements}: We would like to thank Jeanne Scott, Andrei Zelevinsky, Bernard Leclerc,
Jan Schr\"oer, Michael Wemyss, Osamu Iyama, Matthew Pressland, 
Karin Baur and Robert Marsh for helpful discussions and input
at various stages of this project.

\goodbreak\section{Motivating observations}
\label{sec:Jkn} 

\New{We start by noting} that the cluster variables in $\GCA{k}{n}$ are not just homogeneous, 
but multihomogeneous (cf. \cite[p375]{Sc06}).
The argument is basically the same as for a single degree, in which case it is 
also familiar in string theory 
(e.g. \cite[\S6]{FHHU}), where it means that Seiberg duality preserves gauge
anomaly cancellation, that is, the balance equation \eqref{eq:bal} below.
See also \cite[Prop.~3.2]{Grab}.

\begin{lemma}\label{lem:balance}
Every cluster variable of $\GCA{k}{n}$ is multihomogeneous, that is, 
an eigenfunction for the action of the diagonal torus of $\GL_n(\k)$.
\end{lemma}

\newcommand{\inout}{\Delta} 
\newcommand{\multi}{\mathbf{d}} 

\begin{proof}
The proof is inductive, starting from an initial seed of Pl\"ucker coordinates,
which are eigenfunctions by definition, and proceeding by mutation.
The mutation formula for a new variable $x'_m$ in terms of the seed 
$(Q,\underline{x})$ is
\[
   x'_m x_m = \prod_{i\to m} x_i + \prod_{j\from m} x_j
\]
and so, if the cluster variables $x_i$ are eigenfunctions of 
$\GL_n(\k)$-weight $\multi(x_i)\in \ZZ^n$, then 
$x'_m$ will be an eigenfunction of weight $\multi(x'_m)= \inout_m - \multi(x_m)$ if and only if
\begin{equation}\label{eq:bal}
 \sum_{i\to m} \multi(x_i) =  \sum_{j\from m} \multi(x_j),
\end{equation}
with common value $\inout_m$.
Hence, for the induction to proceed, we must also check that the balance equation 
\eqref{eq:bal} holds for some initial seed and is preserved by mutation.

For the initial seed from \cite[\S4]{Sc06}, the equation holds because
every mutable vertex has precisely two incoming and two outgoing arrows;
indeed, every initial exchange relation is a short Pl\"ucker relation \eqref{eq:pluckrel}.
The balance equation \eqref{eq:bal} is unchanged by mutation at any vertex 
not connected to $m$ and also at $m$ itself, since then all arrows simply reverse direction.
For mutation at a vertex $k$ with $k\to m$,
the term $\multi(x_k)$ is effectively replaced by
$-\multi(x'_k) + \sum_{j\to k} \multi(x_j)$, 
which is equal to $\multi(x_k)$.
Almost the same argument applies for mutation at $k\from m$. 
\end{proof}

\New{In the light of this fact, it is natural to look more closely at the $\GL_n(\k)$-weights of the cluster variables,
especially in the finite cluster-type cases.}

First notice that, since $V_{d\omega}$ is a summand of $\Sym^d \Lambda^k(\k^n)$, all of the $\GL_n(\k)$-weights
of $\GCA{k}{n}$ are in the sublattice
\[
 \ZZ^n(k) = \{ x\in \ZZ^n : \text{$k$ divides $\textstyle\sum_i x_i$} \}.
\]
This lattice may be graded by the function $\gr\colon \ZZ^n(k)\to\ZZ$, given by
\[
  \gr(x)= \frac{1}{k} \textstyle\sum_i x_i,
\]
so that all weights $\lambda$ of $V_{d\omega}$ have $\gr(\lambda)=d$.

Inside $\ZZ^n(k)$ we may identify two important sets of vectors
defined in terms of the elementary basis vectors $e_1,\dots,e_n$ of $\ZZ^n$,
i.e. the weights for the fundamental representation $\k^n$ of $\GL_n(\k)$.
\begin{align}
\label{eq:alpha} 
  \alpha_j &= e_{j+1}-e_j\, ,\quad\text{for $j=1,\ldots n-1$,}  \\
\label{eq:beta}  
  \beta_I &= \textstyle\sum_{i\in I} e_i\, , \quad\text{for any $k$-subset $I\sub\intvl{1}{n}$.}
\end{align}
The $\alpha_j$ are the negative simple roots of $\GL_n(\k)$,
while the $\beta_I$ are the weights of the representation $V_\omega=\Lambda^k(\k^n)$.
In particular, $\beta_{\cycint{n}}=e_1+\cdots+e_k$ is the highest weight.

Now, it is also possible to equip $\ZZ^n(k)$ with a quadratic form, given by
\[
  q(x)= \textstyle\sum_i x_i^2 + (2-k) \gr(x)^2,
\]
which is characterised by the fact that $q=2$ on all the roots of $\GL_n(\k)$ and all the $\beta_I$.
In the basis, $\alpha_1,\ldots,\alpha_{n-1},\beta_{\cycint{n}}$,
the quadratic form~$q$ is given by the Cartan matrix
of the graph $J_{k,n}$, obtained by attaching an extra node (for $\beta_{\cycint{n}}$)
to the $k$th node of an 
$A_{n-1}$ graph, i.e. the Dynkin diagram of $\GL_n(\k)$, as follows.
\begin{equation}\label{eq:Jkn}
\begin{tikzpicture}[scale=0.7,baseline=(bb.base)]
\path (0,-0.5) node (bb) {};
\draw[dashed] (0,0)--(5,0);
\draw (2,0)--(2,-1);
\draw (0,0) node {$\bullet$} node[above=3pt] {$1$};
\draw (2,0) node {$\bullet$} node[above=3pt] {$k$};
\draw (5,0) node {$\bullet$} node[above=3pt] {$n-1$};
\draw (2,-1) node {$\bullet$} node[left=3pt] {$n$};
\end{tikzpicture}
\end{equation}
Thus $\ZZ^n(k)$ can be identified with the root lattice $\rootlat(J_{k,n})$
of the associated Kac-Moody algebra.
The grading/degree function $\gr$ is given by the coefficient at the $n$th node,
since $\gr(\alpha_j)=0$ and $\gr(\beta_{\cycint{n}})=1$.

The first thing that indicates the potential significance of this identification is that
$\GCA{k}{n}$ has finite cluster type precisely when $J_{k,n}$ is a Dynkin diagram,
i.e. the corresponding root system is finite.
In addition, $J_{k,n}$ is an affine diagram in precisely the two tame cases 
$\GCA39\isom \GCA69$ and $\GCA48$.

The identification was set up so that the roots of degree 1, i.e. the $\beta_I$,
are precisely the weights of the cluster variables of degree 1, i.e. the
Pl\"ucker coordinates $\Pluck{I}$.
In finite cluster type, we can calculate all the roots of higher degree.
For $k=1,2$ there are no roots, and no cluster variables, of degree $d\geq 2$.
Note that the case $k=1$ is of trivial cluster type, with all $n$ coordinates frozen. 
For $k=3$, the roots of degree 2 are the weights of $\Lambda^6(\CC^n)$,
while for $(k,n)=(3,8)$, the roots of degree 3 are the weights of $\CC^8\otimes \Lambda^8(\CC^8)$.
There are no roots of degree $d\geq 4$.
Comparing this with \cite[Thm 6,7,8]{Sc06}, we have the following.

\begin{observation}\label{obs:clusvar}
In finite cluster type, the weights in $\ZZ^n(k)\isom \rootlat(J_{k,n})$ 
of the cluster variables of $\GCA{k}{n}$ of degree $d$ 
are precisely the roots of degree $d$, for $d>0$.
Furthermore, each root of degree $d$ occurs with multiplicity $d$.
\end{observation}

As partial confirmation, we record the numbers of roots of each degree for $J_{k,n}$ 
and the numbers of cluster variables of $\GCA{k}{n}$ in the following table.
Remember that, in all cases, there are $n$ frozen cluster variables.
\[
\begin{tabular}[c]{l|c|ccc|c|c}
 $k,n$ & $J_{k,n}$ & $d=1$ & $d=2$ & $d=3$  & clu.type & \# clu.var.\\ \hline
$1,n$ & $A_n$ & $n$ & 0 & 0 & $\emptyset$ & $n$ \\
$2,n$ & $D_n$ & $\sfrac{n(n-1)}{2}$ & 0 & 0 & $A_{n-3}$ & $\sfrac{n(n-1)}{2}$\\
$3,6$ & $E_6$ &20 & 1 & 0 & $D_4$ & 22\\ 
$3,7$ & $E_7$ &35 & 7 & 0 & $E_6$ & 49\\ 
$3,8$ & $E_8$ &56 & 28 & 8 & $E_8$ & 136\\ 
\end{tabular}
\]

In these finite type cases, it is possible to calculate all the indecomposables $N$ in $\Sub Q_k$,
which are all rigid. 
Their homogenised cluster characters $\cluschartil{N}$,
as in \cite[\S10]{GLS08},
give all the cluster variables of $\GCA{k}{n}$
except $\Pluck{\cycint{n}}$.
We then find a parallel observation by assigning to each module $N$ in $\Sub Q_k$
an `enhanced' dimension vector, which is a function of the nodes of $J_{k,n}$,
given by the ordinary dimension vector on the $A_{n-1}$ part,
together with $\dim\soc(N)=\deg\cluschartil{N}$ on node~$n$. 

\begin{observation}\label{obs:subQk}
In finite type, 
the enhanced dimension vectors of rigid indecomposable modules
in $\Sub Q_k$ are precisely the roots of $J_{k,n}$ of positive degree, 
with the single exception of the simple root at node $n$. 
Furthermore each root of degree $d$ occurs with multiplicity $d$.
\end{observation}

\New{We will find later, in Remark~\ref{rem:cycling}, an empirical explanation for the multiplicities,
but we currently have no explanation for why the roots appear.}

\begin{remark}\label{rem:young} 
\New{
In the case $d=1$, we may be more explicit 
and also see that Observations~\ref{obs:clusvar} and \ref{obs:subQk} are compatible.
For each Pl\"ucker label $I$, there is a (rigid indecomposable) submodule $N_I$ of $Q_k$ whose dimension vector is
$\beta_I-\beta_{\cycint{n}}$ expressed in the basis $\{\alpha_j\}$ 
and thus, except in the case $N_{\cycint{n}}=0$, 
its enhanced dimension vector is $\beta_I$.
This is a categorification of the familiar bijection between the Young diagrams contained
 in a $k\times (n-k)$ box (corresponding to $Q_k$) and the weights of $\Lambda^k(\k^n)$.
It is a basic example of Geiss-Leclerc-Schr\"oer's general programme and,
in particular, it follows from \cite[\S6.2]{GLS08} that $\cluschartil{N_I}=\Pluck{I}$.
}
\end{remark}

We illustrate Observation~\ref{obs:subQk} and Remark~\ref{rem:young} 
in Figures~\ref{fig:subQ2} and~\ref{fig:edv}.
The first figure shows the Auslander-Reiten quiver of
$\Sub Q_2$ for $\Pi(A_{4})$, depicting the modules by their composition factors.
The top right module is $Q_2$ itself and, in this case, all other \New{indecomposable} modules in $\Sub Q_2$ are
submodules of $Q_2$. 
The projective-injective modules in $\Sub Q_2$ are those along the top and bottom edges.
The other five modules begin and end Auslander-Reiten sequences corresponding
to the five `meshes', one with only one term in the middle.
Note that the picture is more naturally drawn on a M\"obius strip,
with the corresponding modules on the left and right hand diagonal edges identified.

In the second figure, showing the enhanced dimension vectors, we have added the missing simple root
in the obvious gap. 
This completes the fifth mesh in a way that ensures that the enhanced dimension vectors are still additive on Auslander-Reiten sequences.
This strongly suggests that Figure~\ref{fig:edv} should be a picture of the 
Auslander-Reiten quiver of a category that is an `enhancement' of the category $\Sub Q_2$,
which contains one extra indecomposable object and whose Grothendieck group is identified
with $\rootlat(J_{k,n})$.

The main goal of this paper is to show that such a category does indeed exist.

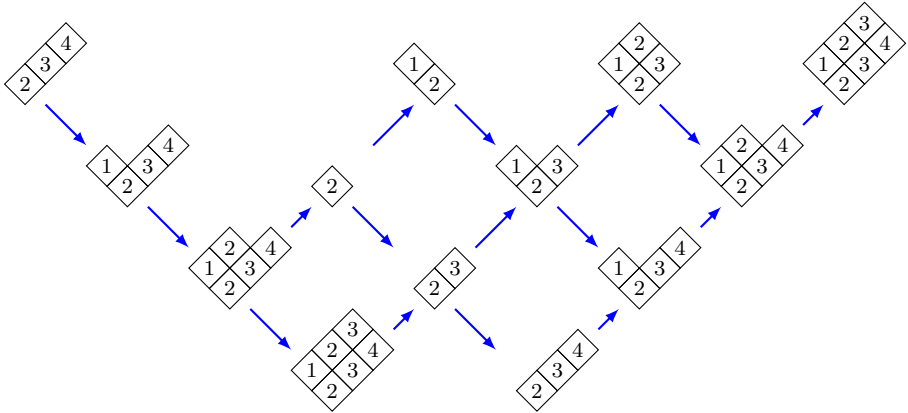
\begin{figure}\centering 
\begin{tikzpicture}[scale=0.27,baseline=(v13.base)]
\newcommand{\projc}{black}
\newcommand{\grow}{5}
\foreach \lab\x/\y/\a/\b/\c/\d/\col in 
 {52/-1/2/1/1/1/1/black, 51/-2/3/0/1/1/1/\projc, 
  53/0/1/1/2/1/1/black, 54/1/0/1/2/2/1/\projc, 13/1/2/0/1/0/0/black, 14/2/1/0/1/1/0/black, 
  15/3/0/0/1/1/1/\projc, 23/2/3/1/1/0/0/\projc, 24/3/2/1/1/1/0/black, 25/4/1/1/1/1/1/black,
  34/4/3/1/2/1/0/\projc, 35/5/2/1/2/1/1/black,  45/6/3/1/2/2/1/\projc}
{ \path (\grow*\x,\grow*\y) node (v\lab) {};
  \foreach \j in {1,...,\b} {\draw[\col] (\grow*\x,\grow*\y-2+2*\j) node (x) {2};
   \draw[\col] (x)++(0,1)--++(1,-1)--++(-1,-1)--++(-1,1)--cycle;}
 \ifthenelse{\a>0}
   {\foreach \j in {1,...,\a} {\draw[\col] (\grow*\x-1,\grow*\y-1+2*\j) node (x) {1};
   \draw[\col] (x)++(0,1)--++(1,-1)--++(-1,-1)--++(-1,1)--cycle; }}
   {}  
  \ifthenelse{\c>0}
   {\foreach \j in {1,...,\c} {\draw[\col] (\grow*\x+1,\grow*\y-1+2*\j) node (x) {3};
   \draw[\col] (x)++(0,1)--++(1,-1)--++(-1,-1)--++(-1,1)--cycle; }}
   {}
 \ifthenelse{\d>0}
   {\foreach \j in {1,...,\d} {\draw[\col] (\grow*\x+2,\grow*\y+2*\j) node (x) {4};
   \draw[\col] (x)++(0,1)--++(1,-1)--++(-1,-1)--++(-1,1)--cycle; }}
   {}
}
\foreach \tx/\ty/\hx/\hy/\ta\ha in
 {-2/3/-1/2/1/2, -1/2/0/1/1/2, 
  0/1/1/0/1/2, 1/2/2/1/1/2, 2/3/3/2/1/2, 
  2/1/3/0/1/2, 3/2/4/1/1/2, 4/3/5/2/1/2 }
 {\path[->,>=latex] (\grow*\tx+\ta,\grow*\ty-\ta)
   edge[blue,thick] (\grow*\hx-\ha,\grow*\hy+\ha);} 
\foreach \tx/\ty/\hx/\hy/\ta\ha in
 {0/1/1/2/3/1, 1/2/2/3/2/1, 1/0/2/1/3/1, 2/1/3/2/2/1,
 3/2/4/3/2/1, 3/0/4/1/3/1, 4/1/5/2/3/1, 5/2/6/3/3/1}
 {\path[->,>=latex] (\grow*\tx+\ta,\grow*\ty+\ta)
   edge[blue,thick] (\grow*\hx-\ha,\grow*\hy-\ha);} 
\end{tikzpicture}
\caption{The sub-category $\Sub Q_2$ of $\md\Pi(A_{4})$.}
\label{fig:subQ2}
\end{figure}

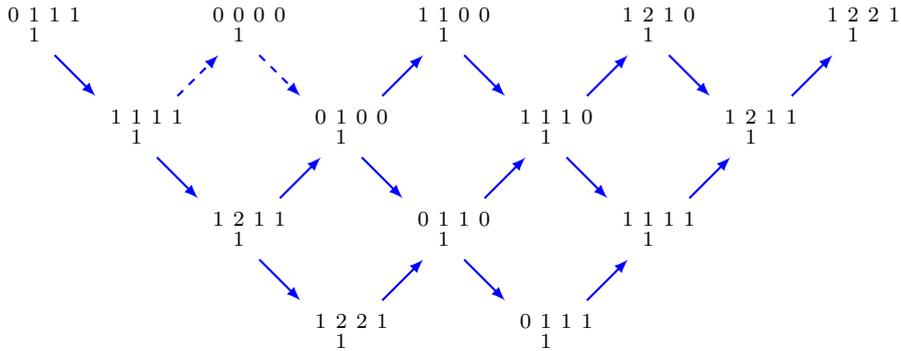
\begin{figure}\centering 
\begin{tikzpicture}[scale=0.27,baseline=(v13.base)]
\newcommand{\projc}{magenta}
\newcommand{\grow}{5}

\foreach \lab\x/\y/\a/\b/\c/\d/\e in 
 {52/-1/2/1/1/1/1/1, 51/-2/3/0/1/1/1/1,
  53/0/1/1/2/1/1/1, 54/1/0/1/2/2/1/1, 13/1/2/0/1/0/0/1, 14/2/1/0/1/1/0/1, 
  15/3/0/0/1/1/1/1, 23/2/3/1/1/0/0/1, 24/3/2/1/1/1/0/1, 25/4/1/1/1/1/1/1,
  34/4/3/1/2/1/0/1, 35/5/2/1/2/1/1/1,  45/6/3/1/2/2/1/1, 12/0/3/0/0/0/0/1}
{ \draw (\grow*\x,\grow*\y) node (v\lab) {\e};
  \draw (v\lab)++(-1,1) node {\a};
  \draw (v\lab)++(0,1) node {\b};
  \draw (v\lab)++(1,1) node {\c};
  \draw (v\lab)++(2,1) node {\d}; }
  
\foreach \tx/\ty/\hx/\hy in
 {-2/3/-1/2, -1/2/0/1, 0/1/1/0, 1/2/2/1,
  2/3/3/2, 2/1/3/0, 3/2/4/1, 4/3/5/2}
 {\path[->,>=latex] (\grow*\tx+1,\grow*\ty-1)
   edge[blue,thick] (\grow*\hx-2,\grow*\hy+2);} 
\foreach \tx/\ty/\hx/\hy in
 {0/1/1/2, 1/2/2/3, 1/0/2/1, 2/1/3/2,
 3/2/4/3, 3/0/4/1, 4/1/5/2, 5/2/6/3}
 {\path[->,>=latex] (\grow*\tx+2,\grow*\ty+2)
   edge[blue,thick] (\grow*\hx-1,\grow*\hy-1);} 
   
 \foreach \tx/\ty/\hx/\hy in {0/3/1/2}
 {\path[->,>=latex] (\grow*\tx+1,\grow*\ty-1)
   edge[blue,thick,dashed] (\grow*\hx-2,\grow*\hy+2);} 
\foreach \tx/\ty/\hx/\hy in {-1/2/0/3}
 {\path[->,>=latex] (\grow*\tx+2,\grow*\ty+2)
   edge[blue,thick,dashed] (\grow*\hx-1,\grow*\hy-1);} 
\end{tikzpicture}
\caption{Enhanced dimension vectors for $\Sub Q_2$.}
\label{fig:edv}
\end{figure}

\goodbreak\section{Describing the category}
\label{sec:cat} 

Let $G=\{\zeta\in \k^* : \zeta^n=1\}$ act on $\k^2$, via 
\begin{equation}\label{eq:Gaction}
  (x,y)\mapsto (\zeta x, \zeta^{-1} y)
\end{equation}
so that $G$ is identified with a finite subgroup of $\SL_2(\k)$.
Then $t=xy$ is a $G$-invariant function and $f=x^k-y^{n-k}$ 
is semi-invariant; geometrically, the $G$ action restricts to the singular curve in $\k^2$ 
with equation $x^k=y^{n-k}$.
Let 
\begin{equation}\label{eq:Rdef} 
  \Rb=\k[x,y]/(f), \quad \Rh=\k[[x,y]]/(f).
\end{equation}
Then $G$ also acts on $\Rb$ and $\Rh$ via \eqref{eq:Gaction}
and we may check that their invariant subrings are
\begin{equation}\label{eq:Zdef}
  \Zb=\k[t]= \Rb^G, \quad \Zh=\k[[t]]= \Rh^G.
\end{equation}
To see this, note that $\k[x,y]^G=\k[x^n,y^n,xy]$, but adding the relation 
$x^k=y^{n-k}$ means that $y^n=(xy)^k$ and $x^n=(xy)^{n-k}$ and so 
$\Rb^G$ may be identified with $\k[xy]$, which 
maps injectively to the quotient $\k[x,y]/(f)$.
In fact, every eigenspace (or isotypic summand) of $\Rb$ under the $G$ action
is a free $\Zb$-module of rank 1. The generators are $1,x,\ldots,x^k,y^{n-k-1}, \ldots, y$.
Hence $\Rb$ is a free $\Zb$-module of rank $n$.
A similar argument applies to $\Rh$.

\begin{remark}\label{rem:Rcomp}
Note that, while $\Zh$ is simply the completion of $\Zb$ with respect to the ideal $(t)\sub \Zb$,
we may regard $\Rh$ as the completion of $\Rb$ in several ways:
either as a $\Zb$-module, i.e. $\Rh = \Rb \otimes_{\Zb} \Zh$,
or with respect to the ideal $(t)\sub \Rb$,
or with respect to the ideal $(x,y)\sub \Rb$.
The last two completions coincide because $(x,y)^n \sub (t) \sub (x,y)$.
\end{remark}

The categories $\md_G(\Rb)$ and $\md_G(\Rh)$ of 
$G$-equivariant finitely generated $\Rb$-modules and $\Rh$-modules, 
respectively, are tautologically equivalent to the 
finitely generated module categories $\md(\Ab)$ and $\md(\Ah)$
for the twisted group rings
\begin{equation}\label{eq:Adef} 
  \Ab=\Rb\tgr G, \quad\Ah=\Rh\tgr G.
\end{equation}
Writing $Z,R,A$ for either $\Zb,\Rb,\Ab$ or $\Zh,\Rh,\Ah$, 
note that $Z$ is precisely the centre of $A$ (because $G$ acts faithfully on $\k^2$)
and that $A$ is a free $Z$-module of rank $n^2$.
We may alternatively describe $\md_G(R)$ as the category of $G\dual$-graded $R$-modules,
where $G\dual$ is the abelian group of linear characters of $G$, 
which is canonically identified with $\ZZ_n$, by setting $\chi_k(\zeta)=\zeta^k$.

Exploiting the identification $\md(A)=\md_G(R)$, we will define
\begin{equation}\label{eq:CM2}
  \CM(A) = \CM_G(R),
\end{equation}
that is, the category of $G$-equivariant (or $G\dual$-graded) Cohen-Macaulay $R$-modules.
As $R$ is a finitely generated $Z$-module, the Cohen-Macaulay $R$-modules
are precisely those which are Cohen-Macaulay over $Z$ 
(e.g. \cite[Ch.~IV, Prop.~11]{Ser1}).
Thus we could also have defined $\CM(A)$ to be the subcategory of $A$-modules
which are Cohen-Macaulay over $Z$, following Auslander \cite[Ch.I, \S7]{Aus78}.
As $Z$ is a PID, this is equivalent to being free over $Z$ and thus $\CM(A)$ 
is also the category of `$A$-lattices' (cf. \cite[Ch.~13]{Sim}).

A further key property of $A$ is that it is an order over $Z$. 
In other words, if $K$ is the field of fractions of $Z$, then
\begin{equation}\label{eq:order}
 A\otimes_{Z} K \isom M_n\bigl(K\bigr)
\end{equation}
More explicitly, $A$ may be canonically identified with a subalgebra of the matrix algebra $M_n(R)$ 
of the following form (illustrated in the case $n=5$, $k=3$),
where the $ij$ entry is the $G$-eigenspace of $R$ with eigenvalue $i-j\in\ZZ_n$.
\begin{equation}\label{eq:xyZmatrix}
\begin{pmatrix}
    Z &   yZ & x^3Z & x^2Z &   xZ \\ 
   xZ &    Z &   yZ & x^3Z & x^2Z \\
 x^2Z &   xZ &    Z &   yZ & x^3Z \\
 x^3Z & x^2Z &   xZ &    Z &   yZ \\
   yZ & x^3Z & x^2Z &   xZ &    Z 
\end{pmatrix}
\end{equation}
We may then think of $A$ as a formal matrix algebra over $Z$ of this form,
subject to the additional relations that $xy=t$ and $x^k=y^{n-k}$.
Alternatively, setting $x=t^{(n-k)/n}$, $y=t^{k/n}$, we can realise
$A$ as a tiled order of index $n$, in the sense of Simson~\cite[\S13]{Sim}.
See also Demonet-Luo~\cite{DL}, in the case $k=2$.

Taking the tiled order point-of-view, as in \eqref{eq:xyZmatrix}, the category $\CM(\Ah)$ 
may be analysed using a covering poset, as in \cite[Ch.~13]{Sim} (see also \cite{RW}).
In particular, the width of this poset is $w=\min(k,n-k)$ and one may thus
obtain the anticipated result that $\CM(\Ah)$ has finite type for $w=2$
and $w=3$, $n=6,7,8$, using Kleiner's criterion \cite[Thm~10.1]{Sim}.

\begin{remark}\label{rem:bartohat}
Much of what we do in this paper is insensitive as to whether we are in the complete case,
when $Z,R,A$ is $\Zh,\Rh,\Ah$, or the non-complete case,
when $Z,R,A$ is $\Zb,\Rb,\Ab$. Indeed, these two cases are closely related by the fact that
\begin{equation}\label{eq:bartohat}
 \Rh = \Rb\otimes_{\Zb} \Zh, \quad \Ah = \Ab\otimes_{\Zb} \Zh.
\end{equation}
Note: here ``$=$'' means that the natural map is an isomorphism and is stronger than ``$\isom$''.
Thus, although $\Rh$ and $\Ah$ are completions of $\Rb$ and $\Ab$
(see Remarks~\ref{rem:Rcomp} \& \ref{rem:Acomp}),
it is simpler to think of them as being related by base change.
For example, \eqref{eq:order} follows automatically for $\Ah$ given the result for $\Ab$.
In addition, there is a comparison functor
\begin{equation}\label{eq:comparison}
 c\colon \md(\Ab) \to \md(\Ah) \colon M\mapsto M^c=M\otimes_{\Zb} \Zh,
\end{equation}
which satisfies
\begin{align}
 \Hom_{\Ah}(M^c,N^c) &= \Hom_{\Ab}(M,N) \otimes_{\Zb} \Zh \\
 \Ext^i_{\Ah}(M^c,N^c) &= \Ext^i_{\Ab}(M,N) \otimes_{\Zb} \Zh
\end{align}
(cf. \cite[Thm~7.11, Exer.~7.7]{Mat}) and also
takes CM-modules to CM-modules.
\end{remark}

\begin{remark}\label{rem:incomplete}
\New{By contrast, there is one key respect in which
it does make a difference whether $A=\Ab$ or $\Ah$.}
Since $\Zh$ is a complete local ring, the Krull-Schmidt Theorem holds in $\md(\Ah)$, but
we do not know whether Krull-Schmidt holds in $\md(\Ab)$.
If it does, then our main result (Theorem~\ref{thm:main}) will also hold
for $\CM(\Ab)$, but then it will follow that $\CM(\Ab)$ and $\CM(\Ah)$
are essentially the same, i.e. the comparison functor $c$ induces a bijection on isomorphism classes.
It can not, of course, be an actual equivalence of categories.

In addition, $\Ah$ is an `isolated singularity' in the sense of
Auslander \cite{Aus86} and so $\CM(\Ah)$ has almost split sequences
and an Auslander-Reiten quiver.
The category $\CM(\Ab)$ might, in practice, be presented by the same quiver, 
but the arrows will not represent irreducible morphisms in the strict sense.
\end{remark}

Another way to describe the algebra $\Ab$,
is as a quotient of the path algebra $\k Q$ of a quiver $Q$, 
namely the McKay quiver of $G\leq\SL_2(\k)$.
In this case, it is the doubled quiver of a simple circular graph $C$
and so we could consider that $\k Q$ is the path algebra of this graph.
More precisely, let $C=(C_0,C_1)$ be the circular graph with vertex set 
$C_0=\ZZ_n=G\dual$ and edge set $C_1=\intvl{1}{n}$, with edge $i$
joining vertices $(i-1)$ and $(i)$. 
Paths on the graph can travel along the edges in either direction and so
we have an associated quiver $Q=Q(C)$ with vertex set $Q_0=C_0$ and
arrows set $Q_1=\{x_a,y_a: a\in C_1\}$ with $x_a\colon (i-1) \to (i)$
and $y_a\colon (i) \to (i-1)$, as illustrated in Figure~\ref{fig:mckayQ} in the case $n=5$.

\begin{figure}
\begin{tikzpicture}[scale=0.90,baseline=(bb.base)]  
\path (0,0) node (bb) {};
\newcommand{\circradius}{1.5cm}
\newcommand{\inradius}{1.2cm}
\newcommand{\outradius}{1.8cm}
\draw[blue,thick] (0,0) circle(\circradius);
\foreach \j in {1,...,5}
{\draw (90-72*\j:\circradius) node[black] {$\bullet$};
 \draw (90-72*\j:\outradius) node[black] {\j};
 \draw (126-72*\j:\inradius) node[black] {\j}; }
\end{tikzpicture}
\qquad\qquad
\begin{tikzpicture}[scale=0.90,baseline=(bb.base)]
\path (0,0) node (bb) {};
\newcommand{\radius}{1.5cm}
\foreach \j in {1,...,5}{
  \path (90-72*\j:\radius) node[black] (w\j) {$\bullet$};
  \path (162-72*\j:\radius) node[black] (v\j) {};
  \path[->,>=latex] (v\j) edge[blue,bend left=25,thick] node[black,auto] {$x_{\j}$} (w\j);
  \path[->,>=latex] (w\j) edge[blue,bend left=20,thick] node[black,auto] {$y_{\j}$}(v\j);
}
\draw (90:\radius) node[above=3pt] {5};
\draw (162:\radius) node[above left] {4};
\draw (234:\radius) node[below left] {3};
\draw (306:\radius) node[below right] {2};
\draw (18:\radius) node[above right] {1};
\end{tikzpicture}
\caption{The circular graph $C$ and double quiver $Q(C)$}
\label{fig:mckayQ}
\end{figure}
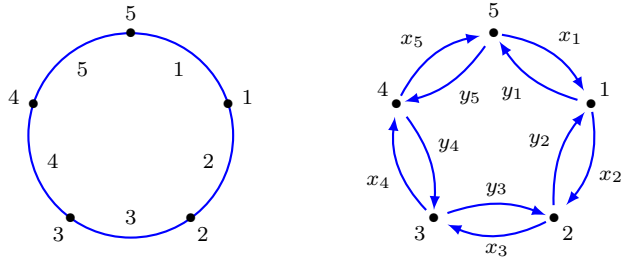

It is familiar from the McKay correspondence 
that $\k[x,y]\tgr G$ is isomorphic to the
preprojective algebra of type $\Atil_{n-1}$, 
that is, the quotient of the path algebra $\k Q$
by the $n$ relations $xy=yx$, one beginning at each vertex.
If we quotient further by the $n$ relations $x^k=y^{n-k}$,
then we obtain $\Rb\tgr G$, that is, $\Ab$.

\begin{remark}\label{rem:Acomp}
As in Remark~\ref{rem:Rcomp}, 
we may regard $\Ah$ as the completion of $\Ab$ in several ways:
either as a $\Zb$-module, or with respect to the ideal $(t)\sub \Ab$,
since $t$ is central,
or with respect to the arrow ideal $(x,y)\sub \Ab$,
since $(x,y)^n \sub (t) \sub (x,y)$.
Hence we may also present $\Ah$
as a quotient of the complete path algebra $\widehat{\k Q}$
by the ideal generated by the relations
$xy-yx$ and $x^k-y^{n-k}$.
\end{remark}

An immediate consequence of this description (for $A=\Ab$ or $\Ah$) is that,
for each vertex $j\in Q_0$, \New{there is an idempotent $e_j\in A$ and
a corresponding indecomposable projective (left) module $P_j=Ae_j$.
In the case $A=\Ah$, knowing Krull-Schimidt means that this is a complete list of (non-isomorphic) 
indecomposable projectives.}

Note that our convention is that representations of the quiver correspond to 
left $A$-modules. 
Right $A$-modules, i.e. left $A\op$-modules, are representations of the 
opposite quiver. 
\New{There are also indecomposable projective right modules $e_jA$.}

\begin{definition}
For any $A$-module $M$ we can define its \emph{rank}
\[
 \rk(M) = \len\bigl( M\otimes_Z K \bigr), 
\]
noting that $A\otimes_Z K\isom M_n\bigl( K\bigr)$, by \eqref{eq:order},
which is a simple algebra.
\end{definition}

Note that rank is additive on short exact sequences and
$\rk(M) = 0$ for any finite-dimensional $A$-module, 
because these are torsion over $Z$.
\New{
Furthermore, 
$\rk_Z(M) = \dim_K \bigl( M\otimes_Z K \bigr) = n\rk(M)$.
Indeed,  for every $1\leq j\leq n$, 
the fact that $xy=t$ implies that $x$ induces an isomorphism 
$e_jM\otimes_Z K\isom e_{j+1}M\otimes_Z K$ and hence
}
\[
  \rk_Z(e_j M) = \rk(M).
\]

For any $A$-module $M$, we have a dual module $M\dual=\Hom_Z(M,Z)$,
which is a right (resp. left) module, when $M$ is a left (resp. right) module.
If $M$ is Cohen-Macaulay, i.e. free over $Z$,
then so is $M\dual$ and $(M\dual)\dual\isom  M$.

\begin{lemma}\label{lem:selfdual}
As either a left or right $A$-module, $A\dual \isom A$.
More precisely,
 \begin{equation}\label{eq:duals}
  (Ae_i)\dual \isom \New{e_{i+k}} A \, ,
  \qquad
  (e_iA)\dual \isom A\New{e_{i-k}}\, .
\end{equation}
\end{lemma}
 
\begin{proof}
Observe first that $Ae_i$ is generated freely over $Z$ by 
\[
   e_i, xe_i, x^2e_i,\ldots, x^{k}e_i=y^{n-k}e_i,\ldots, ye_i
\]
Hence, $(Ae_i)\dual$ is generated over $Z$ by dual generators
\[
   (e_i)\dual, (xe_i)\dual, \ldots, (x^{k}e_i)\dual, \ldots, (ye_i)\dual
\]
satisfying $(xe_i)\dual x=(e_i)^\vee$, etc. 
Then we see that $(Ae_i)^\vee \isom \New{e_{i+k}}A$, 
by identifying $\New{e_{i+k}}$ with $(x^{k}e_i)\dual$.

Since the projective modules are free over $Z$, 
the second identity in \eqref{eq:duals} follows from the first, by taking duals of both sides.
Finally, since $A= \bigoplus_{j\in Q_0} Ae_j$ as a left module, 
we deduce that $A\isom A\dual$, and similarly for $A$ as a right module.
\end{proof}

For example, Figure~\ref{fig:projective+dual} shows $Ae_0$ and 
$(Ae_0)\dual \isom e_3A = \New{A\op e_3}$,
in the case $n=5$, \New{$k=3$}. 
In the figure, each dot in column $i$ represents a $\k$-basis vector in the $Z$-module $e_iAe_0$
(on the left) or $e_3Ae_i$ (on the right).
Columns 0 and 5 should be identified so that the white dots coincide.

\begin{figure}
\begin{tikzpicture}[scale=0.7]
\foreach \j in {0,...,5}
  {\path (\j,4.0) node (a) {$\j$};}
\path (0,3) node (a0) {$\circ$}; \path (5,2) node (a5) {$\circ$}; 
\foreach \v/\x/\y in
  {a1/1/2, a2/2/1, a3/3/0, a4/4/1, b0/0/1, b1/1/0, b2/2/-1, b4/4/-1, b5/5/0, c0/0/-1}
  {\path (\x,\y) node (\v) {$\bullet$};}
\foreach \j in {1,3,5}
  {\path (\j,-1.2) node {$\vdots$};}
\foreach \t/\h in
  {a0/a1, a1/a2, a2/a3, b0/b1, b1/b2, a3/b4, a4/b5}
  {\path[->,>=latex] (\t) edge[blue,thick] node[black,above right=-2pt] {$x$} (\h);}
\foreach \t/\h in
  {a4/a3, a5/a4, b5/b4, a3/b2, a2/b1, a1/b0, b1/c0}
  {\path[->,>=latex] (\t) edge[blue,thick] node[black,above left =-2pt] {$y$} (\h);}
\end{tikzpicture}
\qquad\qquad
\begin{tikzpicture}[scale=0.7]
\foreach \j in {0,...,5}
  {\path (\j,6.8) node (a) {$\j$};}
\path (0,3) node (a0) {$\circ$}; \path (5,4) node (a5) {$\circ$};
\foreach \v/\x/\y in
  {a1/1/4, a2/2/5, a3/3/6, a4/4/5, b1/1/2, b2/2/3, b3/3/4, b4/4/3, b5/5/2, c3/3/2}
  {\path (\x,\y) node (\v) {$\bullet$};}
\foreach \j in {0,2,4}
  {\path (\j,1.6) node {$\vdots$};}
\foreach \t/\h in
  {a0/a1, a1/a2, a2/a3, b1/b2, b2/b3, b3/a4, b4/a5, c3/b4}
  {\path[->,>=latex] (\h) edge[blue,thick] node[black,above left =-2pt] {$x$} (\t);}
\foreach \t/\h in
  {a4/a3, a5/a4, b5/b4, b4/b3, b3/a2, b2/a1, b1/a0, c3/b2}
  {\path[->,>=latex] (\h) edge[blue,thick] node[black,above right=-2pt] {$y$} (\t);}
\end{tikzpicture}
\caption{A projective module and its dual.}
\label{fig:projective+dual}
\end{figure}
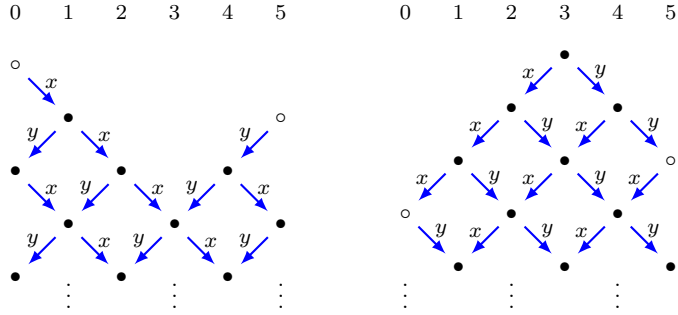

\begin{corollary}\label{cor:gorenstein}
The algebra $A$ is Gorenstein in the sense
of Buchweitz \cite{Bu}, i.e. it is (left and right) noetherian and of finite
 (left and right) injective dimension.
Furthermore
\begin{equation}\label{eq:CM3}
  \CM(A) = \{ M : \Ext_A^i(M,A)=0,\;\text{for $i>0$} \}.
\end{equation}
Hence $\CM(A)$ is a Frobenius category in which the projective-injective objects are the projective $A$-modules.
\end{corollary}

\begin{proof}
Note first that $A$ is noetherian 
because it is finitely generated as a module over a commutative noetherian ring.
Second, because $A$ is a free $Z$-module, the functor
 \[
   \Hom_Z(A,-)\colon \md Z \to \md A
 \]
is exact when interpreted
as a functor to either left or right $A$-modules.
This functor also preserves injectives, \New{as it is right adjoint to the (exact) forgetful functor,} 
and $Z$ is Gorenstein of injective dimension one,
so $A\isom \Hom_Z(A,Z)$ has (left and right) injective dimension one
and thus, in particular, is Gorenstein.

\New{
To deduce \eqref{eq:CM3}, it suffices to observe that, by Lemma~\ref{lem:selfdual},
\[
 \Ext^i_A(M,A) \isom \Ext^i_A(M,\Hom_Z(A,Z))
\]
and that the derived version of the adjunction between $\Hom_Z(A,-)$ and the forgetful functor $A\otimes_A -$
yields
\[
\Ext^i_A(M,\Hom_Z(A,Z)) \isom 
\Ext^i_Z(A\otimes_A M,Z) = 
\Ext^i_Z(M,Z),
\]
because $\Hom_Z(A,-)$ is exact.
Finally, as $Z$ is Gorenstein, $\Ext^i_Z(M,Z)=0$, for all $i>0$, precisely when $M$ is Cohen-Macaulay over $Z$, i.e. in $\CM(A)$.
}

Note that the right-hand-side of \eqref{eq:CM3} is also denoted $GP(A)$
and actually the definition of (maximal) Cohen-Macaulay $A$-modules used in \cite[\S4.2]{Bu}.
It necessarily yields a Frobenius category
\New{in which the projective objects are precisely the projective $A$-modules,
}
(see \cite[Remark~4.8]{Bu} and \cite[\S2.1]{KIWY}).
\New{In our case, the coincidence of projectives also follows directly from the fact that the first syzygy of any $A$-module is a submodule of a free $Z$-module, hence is a free $Z$-module, i.e. in $\CM(A)$ as originally defined.}
\end{proof}

\goodbreak\section{The Main Theorem}
\label{sec:mainthm} 

We now proceed towards our main goal: to relate $\CM(A)$ to the category
$\Sub Q_k$ of Geiss-Leclerc-Schr\"oer \cite{GLS08}.
The main result (Theorem~\ref{thm:main}) and its consequences
only hold in the complete case $A=\Ah$, but the initial results
preceding it still hold in either case.

\begin{lemma}\label{lem:extAZ}
For any $N$ in $\md A$, we have a natural isomorphism
\[
  \Ext_A^1(N,P_n) \isom \Ext_Z^1(e_kN,Z)
\]
\end{lemma}

\begin{proof}
Using \eqref{eq:duals} to write $P_n=Ae_n \isom \Hom_Z(e_kA,Z)$,
we have the derived adjunction
\begin{equation}\label{eq:der.adj}
\Ext_A^1(N,\Hom_Z(e_kA,Z)) \isom \Ext_Z^1(e_kA\otimes_AN,Z),
\end{equation}
noting that $e_kA\otimes_A-$ is exact and so is $\Hom_Z(e_kA,-)$, \New{since $e_kA$ is free over $Z$,}
and so they don't need to be derived themselves.
\end{proof}

As an immediate consequence, if we write $S_j$ for the simple $A$-module 
at vertex $j$, then we have
\begin{equation}\label{eq:extSP}
\dim\Ext_A^1(S_j,P_n)=
  \begin{cases}
    1 & \text{when $j=k$\New{,}}\\
    0 & \text{otherwise\New{.}}
  \end{cases}
\end{equation}

Let $(e_n)$ be the two-sided ideal generated by the idempotent $e_n$.
From the description of $A$ as a quiver algebra,
we see that the quotient algebra $A/(e_n)$ is isomorphic 
to the quotient $B$ of the preprojective algebra $\Pi=\Pi(A_{n-1})$,
by the additional relations $x^k=0=y^{n-k}$.
Thus we have a functor
\begin{equation}\label{eq:trun}
  \trun\colon \md A \to \md B \colon M \mapsto M/(e_n)M
\end{equation}

\begin{remark}\label{rem:subQk}
Recall that, to categorify the cluster algebra structure on the affine
open cell $\N$ of the Grassmannian $\grass{k}{n}$, Geiss-Leclerc-Schr\"oer
\cite{GLS08} studied the \New{Frobenius} subcategory $\Sub Q_k$ of $\md \Pi$
consisting of those modules with socle concentrated at the vertex $k$,
that is, modules isomorphic to a submodule of a direct sum of copies 
of the indecomposable injective $\Pi$-module $Q_k$ with socle $S_k$.
Note that $Q_k$ is actually a $B$-module
and so $\Sub Q_k$ is a subcategory of $\md B$.
Furthermore, \New{one can see by direct computation that}
the \New{indecomposable} projective (and hence also injective) objects in $\Sub Q_k$ are precisely
the \New{indecomposable} projective $B$-modules.

Since $\Sub Q_k$ is extension-closed in $\md \Pi$,
it is certainly extension-closed in $\md B$.
Thus the exact structure on $\Sub Q_k$ used in \cite{GLS08}
is actually inherited from $\md B$ and the syzygy functor 
$\Omega_k=\tau^{-1}$ \cite[Prop~3.4]{GLS08}
is the usual syzygy functor $\Omega$ on $\md B$.
\New{We believe} that $\Sub Q_k$ should coincide with 
$\CM(B)=\{M\in\md B : \Ext_B^i(M,B)=0, \; i>0\}$, 
but we will not need to know this here.
\end{remark}

\New{We henceforth identify $\md B$ with the subcategory of $\md A$ 
consisting of modules $M$ with $e_nM=0$. 
Thus we also consider $\Sub Q_k$ as a (full) subcategory of $\md A$.}

\begin{proposition}\label{prp:tosubQk}
If $M$ is in $\CM(A)$, then $(e_n)M\isom P_n\otimes V$,
where $V$ generates $e_nM$ (freely) over $Z$.
Thus we have a short exact sequence
\begin{equation}\label{eq:ses.trunM}
 0 \to P_n\otimes V \to M\to \trun M \to 0.
\end{equation}
Furthermore, $\trun M$ is in $\Sub Q_k$ and the
restricted functor $\trun\colon \CM(A) \to \Sub Q_k$ is exact
and maps projectives to projectives.
\end{proposition}

\begin{proof}
In general, 
$(e_n)M$ is the image of the natural map 
\[ 
  \eta_M\colon P_n\otimes_Z e_nM \to M,
\]
noting that $Z=\End(P_n)$ and $e_nM=\Hom_A(P_n,M)$. 
But if $M$ is in $\CM(A)$,
then $e_nM$ is a free $Z$ module and so
$P_n\otimes_Z e_nM\isom P_n\otimes V$,
for $V$ as stated.

In particular $P_n\otimes V$ is also in $\CM(A)$ and hence so is $\ker\eta_M$.
However, $e_n\eta_M$ is an isomorphism and so $\ker\eta_M$
vanishes at vertex $n$ and hence is zero.
Thus $\eta_M$ is injective and so $(e_n)M\isom P_n\otimes V$.

Hence we have the required short exact
sequence \eqref{eq:ses.trunM} and
applying the functor $\Hom_A(S_j,-)$ gives a long exact sequence containing
\begin{equation}
 \Hom_A(S_j,M) \to \Hom_A(S_j,\trun M) \to \Ext_A^1(S_j,P_n) \otimes V.
\end{equation}
But $\Hom_A(S_j,M)=0$, for $M$ in $\CM(A)$ and thus 
\eqref{eq:extSP} implies that 
\[
  \Hom_A(S_j,\trun M)=0 \quad\text{for $j\neq k$,}
\]
i.e. the socle of $\trun M$ is concentrated at the vertex $k$, as required.

For the last part, note that $M\mapsto P_n\otimes_Z\Hom(P_n,M)$ is
an exact functor. 
For $M$ in $\CM(A)$, this gives the submodule $(e_n)M$
and so we deduce that $M\mapsto \trun M$ is exact on $\CM(A)$, 
by the Snake Lemma.
Now $\trun\colon \md A\to \md B$ certainly maps projectives to projectives,
\New{but the projectives in $\CM(A)$ are precisely the projectives in $\md A$ 
(see Corollary~\ref{cor:gorenstein}) and 
the projectives in $\Sub Q_k$ are precisely the projectives in $\md B$ (see Remark~\ref{rem:subQk}).}
Thus the restricted functor also preserves projectives.
\end{proof}

\begin{lemma}\label{lem:alpiso}
For any $N$ in $\Sub Q_k$, the inclusion
\[
  \iota_N\colon \Hom_A(S_k,N) \to \Hom_Z(S_k,e_kN)
\]
is an isomorphism.
Furthermore, the natural map
\[
 \alpha_N\colon \Hom_A(S_k,N) \to \Ext_A^1(S_k,P_n) \otimes \Ext_A^1(N,P_n)^* 
\]
is an injection.
\end{lemma}

\begin{proof}
For the first part, let $h\colon N\to J$ be the injective hull of $N$ 
\New{in $\md \Pi$.
As $N$ is in $\Sub Q_k$, we know
that $J\isom Q_k^{\oplus m}$, where $m=\dim \soc N$.}
Then we have a commuting square
\[
\begin{tikzpicture}[scale=0.8,
 arr/.style={black, -angle 60}]
\path (0,2) node (a0) {$\Hom_A(S_k,N)$};
\path (5,2) node (a1) {$\Hom_Z(S_k,e_kN)$};
\path (0,0) node (b0) {$\Hom_A(S_k,J)$};
\path (5,0) node (b1) {$\Hom_Z(S_k,e_kJ)$};
\path[arr] (a0) edge node[auto]{$\iota_N$} (a1);
\path[arr] (b0) edge node[auto]{$\iota_J$} (b1);
\path[arr] (a0) edge node[auto]{$h_*$} (b0);
\path[arr] (a1) edge (b1);
\end{tikzpicture}
\]
in which all maps are \emph{a priori} inclusions.
However, $h_*$ is an isomorphism, because $h$ is the injective hull
and $\iota_J$ is an isomorphism, because $\iota_{Q_k}$ is an isomorphism;
\New{indeed, both $Q_k$ (as an $A$-module) and $e_kQ_k$ (as a $Z$-module) have socle $S_k$.}
Hence $\iota_N$ is an isomorphism as required.

For the second part, note that the isomorphism of Lemma~\ref{lem:extAZ},
expanded as \eqref{eq:der.adj}, is functorial in $N$ and hence,
because $\iota_N$ is induced by the functor $e_kA\otimes_A-$, 
we get an identification of the product
\[
 \Hom_A(S_k,N)\otimes \Ext_A^1(N,P_n) \to  \Ext_A^1(S_k,P_n)
\]
with the product
\[
 \Hom_Z(S_k,e_kN)\otimes \Ext_Z^1(e_kN,Z) \to \Ext_Z^1(S_k,Z)
\]
and thus of $\alpha_N$ with the map
\[
 \Hom_Z(S_k,e_kN)\to \Hom_k(\Ext_Z^1(e_kN,Z),\Ext_Z^1(S_k,Z)).
\]
But $M\mapsto \Ext^1_Z(M,Z)$ is a duality (i.e. a contravariant involution) 
on finite length $Z$-modules and so in fact this map gives an isomorphism
\[
 \Hom_Z(S_k,e_kN)\isom \Hom_Z(\Ext_Z^1(e_kN,Z),\Ext_Z^1(S_k,Z)),
\]
which in particular yields the required injectivity.
\end{proof}

From now until the end of the section we must have $A=\Ah$.

\begin{theorem}\label{thm:main}
The exact functor 
\[
  \trun\colon \CM(A) \to \Sub Q_k
\]
is a quotient by the ideal generated by $P_n$.
Moreover, for any $N$ in $\Sub Q_k$ there is 
(up to isomorphism) a unique minimal $M$ in $\CM(A)$ with $\trun M\isom N$,
where ``minimal'' means that $M$ has no summand isomorphic to $P_n$.
Such a minimal $M$ satisfies
\begin{equation}\label{eq:rk=dimsoc}
 \rk(M) =\dim\soc \trun M .
\end{equation}
\end{theorem}

\begin{proof}
For any $N$ in $\Sub Q_k$, we may form an extension
\begin{equation}\label{eq:ses.N}
 0 \to P_n\otimes V \to M\to N \to 0,
\end{equation}
determined by a classifying element
$\beta\colon V^* \to \Ext^1(N,P_n)$.
For any simple $A$-module $S$, 
applying $\Hom_A(S,-)$ to \eqref{eq:ses.N} yields
\[
 0\to \Hom(S,M)\to \Hom(S,N)\to \Ext^1(S,P_n)\otimes V .
\]
For $S\not\isom S_k$, we have $\Hom(S,N)=0$ and so $\Hom(S,M)=0$. 
For $S=S_k$, the second map is the composite
\[
 \Hom(S_k,N) \lra{\alpha_N}
 \Ext^1(S_k,P_n)\otimes\Ext^1(N,P_n)^* \lra{1\otimes \beta^*}
 \Ext^1(S_k,P_n)\otimes V  
\]
where $\alpha_N$ is the map of Lemma~\ref{lem:alpiso},
which is an injection. 
Hence, provided we choose $\beta$ so that $(1\otimes \beta^*)\circ \alpha_N$
remains injective, we will have that $\Hom(S_k,M)=0$ also and so $M$ is
torsion-free and hence in $\CM(A)$.

Now, applying $\Hom_A(P_n,-)$ to \eqref{eq:ses.N} yields an isomorphism 
$Z\otimes V\to e_nM$ and hence an identification of $N$ with $\trun M$,
by comparison with \eqref{eq:ses.trunM}.
Thus, the functor $\trun\colon \CM(A) \to \Sub Q_k$ 
is essentially surjective. 

On the other hand, we know from Proposition~\ref{prp:tosubQk} that any
$M$ in $\CM(A)$ is an extension
\begin{equation}\label{eq:MtotrunM}
 0 \to P_n\otimes V \to M\to \trun M \to 0,
\end{equation}
which will be classified by some 
$\beta_M\colon  V^* \to \Ext^1(\trun M,P_n)$
for which 
\[
 (1\otimes \beta_M^*)\circ \alpha_{\trun M}\colon 
\Hom(S_k,\trun M) \lra{} \Ext^1(S_k,P_n)\otimes V  
\]
is an injection. Thus certainly
\begin{equation}\label{eq:rkbound}
 \rk(M)=\dim V \geq \dim\Hom_A(S_k,\trun M) = \dim\soc \trun M. 
\end{equation}
In particular, this means that any $M$ satisfying \eqref{eq:rk=dimsoc}
is minimal.

If $\ker \beta_M= K^*\neq0$, then we can split of a summand of 
$P_n\otimes K$ from the extension, so in particular $M$ is not minimal.
So, suppose that $\beta_M$ is injective and choose a subspace $W^*\leq V^*$,
with $\dim W = \dim\soc \trun M$ 
such that the induced map
\[
 (1\otimes \beta_M^*)\circ \alpha_{\trun M}\colon 
\Hom(S_k,\trun M) \lra{} \Ext^1(S_k,P_n)\otimes W  
\]
is an isomorphism. Then, setting $U=\ker (V\to W)$, we have the
following diagram of four short exact sequences (with zeroes omitted)
\[
\begin{tikzpicture}[scale=0.8,
 arr/.style={black, -angle 60}]

\path (0,0) node(a0) {$P_n\otimes W$};
\path (0,2) node (b0) {$P_n\otimes  V$};
\path (0,4) node (c0) {$P_n\otimes  U$};
\path (3,0) node(a1) {$M'$};
\path (3,2) node (b1) {$M$};
\path (3,4) node (c1) {$P_n\otimes  U$};
\path (5.7,0) node(a2) {$\trun M$};
\path (5.7,2) node (b2) {$\trun M$};
\path[arr] (b0) edge (a0);
\path[arr] (c0) edge (b0);
\path[arr] (b1) edge (a1);
\path[arr] (c1) edge (b1);
\path[arr] (b2) edge node[auto] {$\isom$} (a2);
\path[arr] (a0) edge (a1);
\path[arr] (a1) edge (a2);
\path[arr] (b0) edge (b1);
\path[arr] (b1) edge (b2);
\path[arr] (c0) edge node[auto] {$\isom$} (c1);
\end{tikzpicture}
\]
But $P_n$ is injective in $\CM(A)$ and so, 
since $M'$ is in $\CM(A)$, the middle vertical sequence splits.
Thus $M$ is only minimal when $W=V$, in which case it
satisfies \eqref{eq:rk=dimsoc}.

The above argument also shows that every minimal $M$ in $\CM(A)$ is a summand
of the universal extension $\widehat{N}$ of $N=\trun M$, formed by taking 
$\beta$ to be the identity map, and further that 
$\widehat{N}\isom M\oplus (P_n\otimes U)$,
for some $U$ of determined dimension.
Hence, any two such $M$ are isomorphic,
by the Krull-Schmidt Theorem.

Now, it is straightforward to see that the kernel of the functor
$\trun$ is precisely the ideal of maps factoring through $P_n$, because
firstly $\trun P_n=0$, while the functoriality of 
\eqref{eq:MtotrunM} implies that
any map killed by $\trun$ is in this ideal.

Thus, to see that the functor is the claimed quotient, 
it remains to show that every map $\trun M_1\to \trun M_2$,
lifts to a map $M_1\to M_2$.
It is sufficient to consider that case when both 
$M_i$ are minimal, and then, since such are summands of
the universal extension, with only $P_n$s as complements,
it is sufficient to consider the case of the universal extension.
But the universal extension is formed by taking a cone in the 
derived category on the natural map 
$\trun M_i\to P_n[1]\otimes \Hom(\trun M_i,P_n[1])^*$,
so any map $\phi\colon \trun M_1\to \trun M_2$ does lift
(non-uniquely) to a morphism of extensions
\[
\begin{tikzpicture}[scale=0.7,
 arr/.style={black, -angle 60}]
\path (0,2.4) node (a0) {$P_n\otimes V_1$};
\path (3,2.4) node (a1) {$\widehat{\trun M_1}$};
\path (6,2.4) node (a2) {$\trun M_1$};
\path (0,0) node (b0) {$P_n\otimes V_2$};
\path (3,0) node (b1) {$\widehat{\trun M_2}$};
\path (6,0) node (b2) {$\trun M_2$};
\path (-2.3,2.4) node (ax) {$0$};
\path (7.9,2.4) node (ay) {$0$};
\path (-2.3,0) node (bx) {$0$};
\path (7.9,0) node (by) {$0$};
\path[arr] (a0) edge node[auto]{$1\otimes \phi_*$} (b0);
\path[arr] (a1) edge node[auto]{$\widehat{\phi}$} (b1);
\path[arr] (a2) edge node[auto]{$\phi$} (b2);
\path[arr] (a0) edge (a1);
\path[arr] (a1) edge (a2);
\path[arr] (b0) edge (b1);
\path[arr] (b1) edge (b2);
\path[arr] (ax) edge (a0);
\path[arr] (a2) edge (ay);
\path[arr] (bx) edge (b0);
\path[arr] (b2) edge (by);
\end{tikzpicture}
\]
where $V_i=\Ext^1(\trun M_i,P_n)^*$, which is functorial in
$\trun M_i$, so that $\phi$ induces a map $\phi_*\colon V_1\to V_2$.
\end{proof}

\begin{corollary}\label{cor:stable}
The functor $\trun$ induces a triangle equivalence of stable categories
\[
 \underline{\trun}\colon \underline{\CM}(A) \to \underline{\Sub}\, Q_k
\]
and hence also an isomorphism of Ext groups
\[
  \Ext^1_A(X,Y) \isom \Ext^1_B(\trun X,\trun Y).
\]
\end{corollary}

\begin{proof}
As $\trun$ is exact and maps projectives to projectives (Proposition~\ref{prp:tosubQk}), it induces a 
triangle morphism $\underline{\trun}$ between the corresponding stable categories.
Since $\trun$ is a quotient by some projective (Theorem~\ref{thm:main}),
$\underline{\trun}$ is full and faithful and a bijection on isomorphism classes,
so it is a triangle equivalence (cf. \cite{Hap}).
Since $\Ext^1(X,Y)\cong\SHom(\Omega X,Y)$, in both $\CM(A)$ and $\Sub Q_k$
(see Remark~\ref{rem:subQk} for the latter),
we deduce that $\trun$ induces an isomorphism on $\Ext^1$.
\end{proof}

\begin{remark}\label{rem:main}
The functor $\trun\colon \CM(A) \to \Sub Q_k$ of Theorem~\ref{thm:main}
provides a one-to-one correspondence between the indecomposable modules
in $\CM(A)$, other than $P_n$, and the indecomposable modules in $\Sub Q_k$
and furthermore this restricts to a correspondence between
the indecomposable projectives in $\CM(A)$, other than $P_n$,
and the indecomposable projectives in $\Sub Q_k$.
Since $\trun$ induces an isomorphism on $\Ext^1$,
it also gives a one-to-one correspondence between
non-projective rigid indecomposables.
\end{remark}

\begin{remark}\label{rem:CTO}
Recall that, \New{in a stably 2-Calabi-Yau category}, a cluster tilting object $T$ (\aka~maximal 1-orthogonal object)
is one for which 
\[
   \add(T) = \{X : \Ext^1(T,X)=0 \}.
\]
By \cite[Thm.~II.1.8(a)]{BIRS}, 
if the category has some cluster tilting object,
then this condition is equivalent to the \emph{a priori} 
weaker condition of being maximal rigid. 
It is known, by \cite[Prop.~7.4]{GLS08}, that $\Sub Q_k$ does have at least one cluster tilting object and hence so does $\CM(A)$ as these are stably equivalent.

It is also known, by \cite[Prop.~7.1]{GLS08}, that any maximal rigid object in  $\Sub Q_k$
has $k(n-k)$ non-isomorphic indecomposable summands
and hence any maximal rigid object in $\CM(A)$
has $k(n-k)+1$ non-isomorphic indecomposable summands.
Such objects are called `complete rigid'.

\New{As a consequence, mutation of cluster tilting objects in $\CM(A)$ works
in the same way as in $\Sub Q_k$, as described in \cite[\S8]{GLS08}, that is, 
we may replace any non-projective rigid indecomposable $X$ in a complete rigid object $X\oplus T$
by a unique alternative non-projective rigid indecomposable $Y$ so that $Y\oplus T$ is also a complete rigid object
and furthermore $\dim \Ext^1(X,Y)=1$.
Thus the functor $\trun\colon \CM(A) \to \Sub Q_k$ gives a one-to-one correspondence between
cluster tilting objects in $\CM(A)$ and cluster tilting objects in $\Sub Q_k$ that is compatible with mutation.
}
\end{remark}

\goodbreak\section{Rank one modules}
\label{sec:rkone} 

\newcommand{\rkone}[1]{L_{#1}}

In this section, we return to the situation where $Z,R,A$ denote 
either $\Zh,\Rh,\Ah$ or $\Zb,\Rb,\Ab$.
We study the simplest modules in $\CM(A)$, namely those with rank one.
Recall that rank is additive on short exact sequences, and so, in particular, on direct sums.
Hence any rank one CM-module is necessarily indecomposable,
because all summands of CM-modules are CM-modules and non-zero CM-modules have positive rank. 

In fact, we can classify all rank one modules and see that they are in canonical
one-to-one correspondence with the Pl\"ucker coordinates $\Pluck{I}$ in $\GCA{k}{n}$.

\begin{definition}\label{def:rkone}
For any $k$-subset $I\sub C_1$, define a rank one module
$\rkone{I}$ in $\CM(A)$ as follows.
For $i\in C_0$, set $e_i\rkone{I}=Z$ and, for $a\in C_1$, set 
\begin{align*}
 & \text{ $x_a\colon Z\to Z$ to be (multiplication by) $1$, if $a\in I$, or $t$, if $a\not\in I$,}\\
 & \text{ $y_a\colon Z\to Z$ to be (multiplication by) $t$, if $a\in I$, or $1$, if $a\not\in I$.}
\end{align*}
It is clear that $\rkone{I}$ is a free $Z$-module and that the relations $xy=yx=t$ hold.
We can also check that $x^k=y^{n-k}$ (see the proof of Proposition~\ref{pr:rkone}),
so $\rkone{I}$ is indeed in $\CM(A)$.
It is also clear that the image, under the comparison functor \eqref{eq:comparison},
of $\rkone{I}\in \CM(\Ab)$, defined with $Z=\Zb$, 
is $\rkone{I}\in \CM(\Ah)$, defined with $Z=\Zh$.
\end{definition}

For example, the module $\rkone{124}$ (in the case $n=5$, $k=3$) 
is pictured in Figure~\ref{fig:L124}.
As in Figure~\ref{fig:projective+dual},
each dot in column $i$ represents a $\k$-basis vector in the $Z$-module $e_i\rkone{124}$
and columns 0 and 5 should be identified so that the white dots coincide.
Note: the projective module $Ae_0$ in Figure~\ref{fig:projective+dual} is $\rkone{123}$.

\begin{figure}\centering
\begin{tikzpicture}[scale=0.7]
\foreach \j in {0,...,5}
  {\path (\j,3.5) node (a) {$\j$};}
\path (0,3) node (a0) {$\circ$}; \path (5,2) node (a5) {$\circ$}; 
\foreach \v/\x/\y in
  {z2/3/2, a1/1/2, a2/2/1, a3/3/0, a4/4/1, 
   b0/0/1, b1/1/0, b2/2/-1, b4/4/-1, b5/5/0, c0/0/-1}
  {\path (\x,\y) node (\v) {$\bullet$};}
\foreach \j in {1,3,5}
  {\path (\j,-1.2) node {$\vdots$};}
\foreach \t/\h in
  {z2/a4, a0/a1, a1/a2, a2/a3, b0/b1, b1/b2, a3/b4, a4/b5}
  {\path[->,>=latex] (\t) edge[blue,thick] node[black,above right=-2pt] {$x$} (\h);}
\foreach \t/\h in
  {z2/a2, a4/a3, a5/a4, b5/b4, a3/b2, a2/b1, a1/b0, b1/c0}
  {\path[->,>=latex] (\t) edge[blue,thick] node[black,above left =-2pt] {$y$} (\h);}
\end{tikzpicture}
\caption{The rank one module $\rkone{124}$.}
\label{fig:L124}
\end{figure}
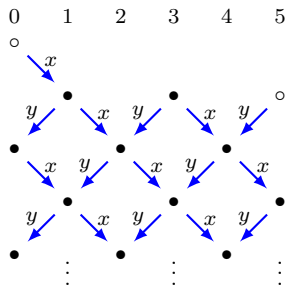

\begin{proposition}\label{pr:rkone}
Every rank one module in $\CM(A)$ is isomorphic
to $\rkone{I}$ for some (unique) $k$-subset $I\in C_1$. 
\end{proposition}

\begin{proof}
For any rank one module $M$ in $\CM(A)$, each piece $e_iM\isom Z$,
where the choice of isomorphism is determined up to $\Aut(Z)$.
Comparing two adjacent pieces $e_iM$ and $e_{i+1}M$, the relation
$xy=t$ implies that, after rescaling just the isomorphism $e_{i+1}M\isom Z$,
we have either $(x,y)=(1,t)$ or $(x,y)=(t,1)$.
Thus starting from $e_0M$ and working around to $e_nM$, 
the module $M$ is characterised by this choice for each pair of arrows $(x,y)$
together with an element $\lambda\in\Aut(Z)$ 
giving the composite $Z\isom e_0M=e_nM\isom Z$.

However, the additional relation $x^k=y^{n-k}$ implies that $\lambda=1$ 
and furthermore that the number of times $x_a=1$ is precisely $k$.
This is necessary (and sufficient) to ensure that, 
for any choice of paths $y^{n-k}$ and $x^k$ with the same head and tail, 
the number of $a$ such that $y_a=t$ (i.e. $x_a=1$)
in $y^{n-k}$ is the same as the number of $a$ such that $x_a=t$ (i.e. $x_a\neq1$) in $x^k$.
Note that the two paths $y^{n-k}$ and $x^k$ involve complementary sets of edges in $C_1$.

Thus the module $M$ is isomorphic to $\rkone{I}$, for some $k$-subset $I\sub C_1$.
Furthermore, it is clear that $\rkone{I}\not\isom\rkone{J}$, for $I\neq J$,
because the restrictions of the modules to a single edge in $I\setminus J$
are not isomorphic.
\end{proof}

\begin{example}
In the case $k=2$, Demonet-Luo~\cite[Thm.~3.1]{DL} have shown that
the indecomposable objects in $\CM(A)$
are precisely the rank one modules $\rkone{ij}$.
Indeed they even show that $\CM(\Ab)$ satisfies Krull-Schmidt in this case.

The Auslander-Reiten quiver of $\CM(\Ah)$ is shown in Figure~\ref{fig:ARqG25}, 
in the case $n=5$.
For general $n$, these A-R quivers may be said to categorify the
Coxeter-Conway frieze patterns associated with triangulations of an $n$-gon.
Note that $ij=ji$,
so the left and right sides are identified by a glide reflection,
as in Figures~\ref{fig:subQ2} and~\ref{fig:edv}.

The A-R quiver of the stable category $\underline{\CM}(\Ah)$ is obtained by deleting 
the projective-injective objects $i(i+1)$ along the edges
and we thus recognise $\underline{\CM}(\Ah)$ as the cluster category of type $A_{n-3}$
(cf. Caldero-Chapoton-Schiffler \cite[Thm~5.2]{CCS}).
\end{example}

\begin{figure}\centering
\begin{tikzpicture}[scale=0.9]
\foreach \v/\x/\y in
  {12/0/3, 13/1/2, 14/2/1, 15/3/0, 
   23/2/3, 24/3/2, 25/4/1, 21/5/0, 34/4/3, 35/5/2, 31/6/1, 45/6/3, 41/7/2, 51/8/3}
  {\path[black] (\x,\y) node (x\v) {$\v$};} 
\foreach \t/\h in
  { 12/13, 13/14, 14/15, 23/24, 24/25, 25/21, 34/35, 35/31, 
   13/23, 14/24, 24/34, 15/25, 25/35, 35/45, 21/31,45/41,31/41,41/51}
  {\path[->,>=latex] (x\t) edge[blue,thick] (x\h);}  
\foreach \x/\y in
  {3/1, 5/1, 2/2, 4/2,6/2}
  {\path[black] (\x,\y) node {$- -$};} 
\end{tikzpicture}
\caption{The Auslander-Reiten quiver of $\CM(\Ah)$ for $\grass25$.}
\label{fig:ARqG25}
\end{figure}

\begin{remark}\label{rem:rkonehoms}
It is straightforward to describe the space of homomorphisms between 
rank one modules, because such a homomorphism is given by a set of 
$Z$-linear maps $\theta_i\colon Z\to Z$, for each $i\in C_0$,
subject to the constraint of commuting with the action of $x$ and $y$.
Thus a basis for
$\Hom(L_I,L_J)$ is given by $t^\alpha=\bigl( t^{\alpha_i}\bigr)_{i\in C_0}$
for exponent vector $\alpha\in \NN^{C_0}$ 
satisfying
\[
  \alpha_{ha}-\alpha_{ta} = 
\begin{cases} 1 & \text{if $a\in J\setminus I$,}\\
  -1 & \text{if $a\in I\setminus J$,}\\
  0 & \text{otherwise.}
\end{cases}
\]
In other words, the exponent vectors of the basis are given by the set
\begin{equation}\label{eq:homdef}
 \hom(I,J) = \{ \alpha\in \NN^{C_0} : d\alpha = \eps_J - \eps_I \},
\end{equation}
where $d\colon \ZZ^{C_0}\to \ZZ^{C_1}$ is the coboundary map of
the cochain complex of $C$
and $\eps_I\in \ZZ^{C_1}$ is the characteristic function of $I$. 
Since any two solutions of \eqref{eq:homdef} differ by a constant,
we see that $\Hom(L_I,L_J)$ is a free rank one $Z$-module, 
generated by a homomorphism $t^\alpha$ with minimal exponent, that is,
for which $\alpha\in \NN^{C_0}$ has at least one zero component.
\end{remark}

\begin{definition}\label{def:crossing}
Two $k$-subsets $I,J$ of $\intvl{1}{n}$ are 
\emph{non-crossing} \cite[Def.~3]{Sc06}
(\aka~\emph{weakly separated} \cite{LZ98,OPS,Sc05}) 
if there are no cyclically ordered $a,b,c,d$ with
$a,c\in I\setminus J$ and $b,d\in J\setminus I$.
If such $a,b,c,d$ do occur, then $I$ and $J$ are \emph{crossing}.
\end{definition}

As expected, i.e. since we know it holds in the stable category by \cite{CCS}, 
this property is equivalent to $\Ext$-vanishing
in the categorification (see also \cite[Prop.~3.11]{DL} in the case $k=2$).

\begin{proposition}\label{prop:nc}
Let $I,J$ be $k$-subsets of $\intvl{1}{n}$.
Then $\Ext^1(\rkone{I},\rkone{J})=0$ if and only if $I$ and $J$ are non-crossing.
\end{proposition}

\newcommand{\shHom}[2]{\Hom({#1},{#2})}
\newcommand{\shExt}[2]{\Ext^1({#1},{#2})}

\begin{proof}
We compute $\Ext^1(\rkone{I},\rkone{J})$ using a projective presentation of $\rkone{I}$
\[
\bigoplus_{v\in V} P_v \lra{D} \bigoplus_{u\in U} P_u \to \rkone{I} \to 0,
\]
where $U=\{ u\not\in I : u+1 \in I\}$ and $V=\{ v \in I : v+1 \not\in I\}$.
Note that $U$ and $V$ are disjoint sets with the same number of elements, 
which alternate in the cyclic order.
This number is $m+1$, where $m=\rk(\Omega \rkone{I})$
and $\Omega \rkone{I}=\img D$ is the first syzygy.
The $(m+1)\times(m+1)$ matrix $D=(d_{vu})$ is supported on just two (cyclic) 
diagonals, as it has non-zero entries only
when $u,v$ are adjacent in $U\cup V$.
Indeed
\begin{equation}\label{eq:Dcoeffs}
 d_{vu} = \begin{cases}
 x^{v-u} & \text{when $u$ precedes $v$,}\\
 -y^{u-v} & \text{when $u$ follows $v$,}\\
  0 & \text{otherwise.}
 \end{cases}
\end{equation}
Applying $\Hom(-,\rkone{J})$ yields
\[
\begin{tikzpicture}[scale=0.8,
 arr/.style={black, -angle 60}]
\path (0.2,3) node (b0) {$\displaystyle\bigoplus_{u\in U}^{\phantom{U}} \shHom{P_u}{\rkone{J}}$};
\path (5,0.75) node(a1) {$\displaystyle\bigoplus_{v\in V} \shHom{P_v}{\rkone{J}}$};
\path (5,3) node (b1) {$\shHom{\Omega \rkone{I}}{\rkone{J}}$};
\path (5,4.25) node (c1) {$0$};
\path (9.1,3) node(b2) {$\shExt{\rkone{I}}{\rkone{J}}$};
\path (11.4,3) node (by) {$0$};
\path[arr] (b0) edge (b1);
\path[arr] (b1) edge (b2);
\path[arr] (b2) edge (by);
\path[arr] (c1) edge (b1);
\path[arr] (b1) edge (a1);
\path[arr] (b0) edge node[auto] {$D^*$} (a1);
\end{tikzpicture}
\]
where matrix $D^*=(d^*_{vu})$ is given by
\begin{equation}\label{eq:D*coeffs}
 d^*_{vu} = \begin{cases}
 t^a & \text{$a=\# [u,v)\setminus J$, when $u$ precedes $v$,}\\
 -t^b & \text{$b=\# J\cap [v,u)$, when $u$ follows $v$,}\\
  0 & \text{otherwise.}
 \end{cases}
\end{equation}
Note that $I$ is precisely the union of the intervals $[u,v)$ 
in the first case, and hence the complement of $I$ is the union
of the intervals $[v,u)$ 
in the second case.
Now $\Ext^1(\rkone{I},\rkone{J})$ is a finite length $Z$-module and
$\Hom(\Omega \rkone{I},\rkone{J})$ is a free $Z$-module of rank $m$.
Hence 
\[
  \Ext^1(\rkone{I},\rkone{J})=0
 \quad\iff\quad 
 \rk_{\k} (D^*_{|t=0})=m.
\]
However, the form of $D$ means that
every $m\times m$ minor of $D^*$ has precisely one term,
which consists of a product of two intervals, one on each diagonal.
One such interval may be empty if the other contains all but one entry.

Thus, all $m\times m$ minors of $D^*_{|t=0}$ will vanish precisely when 
each diagonal contains at least two zeroes and the zeroes on each diagonal
do not occur in non-overlapping intervals.
Since these zeroes come from strictly positive powers of $t$ in
\eqref{eq:D*coeffs}, this occurs precisely when $I$ and $J$ are crossing.
\end{proof}

\begin{remark}
As conjectured by Scott \cite{Sc05} and proved by Oh-Postnikov-Speyer \cite{OPS},
any maximal set $\maxNC$ of non-crossing subsets has $k(n-k)+1$ elements.
Hence Proposition~\ref{prop:nc} and Remark~\ref{rem:CTO} imply that
$T_\maxNC = \bigoplus_{J\in\maxNC} \rkone{J}$
is a complete rigid object in $\CM(\Ah)$ and thus a cluster tilting object.
\end{remark}

\goodbreak\section{Graded modules and profiles}
\label{sec:genfilt} 

\newcommand{\Mlim}{M_*}
\newcommand{\QG}{Q}
\newcommand{\QGam}{\widetilde{Q}}
\newcommand{\compl}{\mathbb{F}}

The symmetry group $G$ of the ring $\Rb=\k[x,y]/(x^k-y^{n-k})$,
as introduced in Section~\ref{sec:cat},
is contained in the larger symmetry group 
\[
  \Gamma=\{(s,t)\in (\k^*)^2 : s^k=t^{n-k} \}.
\]
Indeed, $G=\Gamma\cap \SL_2(\k)$ and so there are dual short exact sequences
\begin{gather*}
 0 \to G \to \Gamma \lra{\det} \k^* \to 0, \\
 0 \to \ZZ \to \Gamma\dual \lra{\;\rho\;} G\dual \to 0.
\end{gather*}
Thus the McKay quiver $\QGam$ of $\Gamma$ (with vertex set $\Gamma\dual$)
may be visualised as a `lattice cylinder',
which is a $\ZZ$-cover $\rho\colon \QGam \to \QG$ of the McKay quiver of $G$.
For example, when $n=5$, $k=3$, we obtain a cylindrical quiver covering the circular
quiver in Figure~\ref{fig:mckayQ}
by matching the white dots to identify the left and right columns in
the infinite strip depicted in Figure~\ref{fig:covering}.
Note: despite the visual similarity, the upper part of this figure depicts a quiver,
whereas Figures~\ref{fig:projective+dual} and~\ref{fig:L124} depict modules.

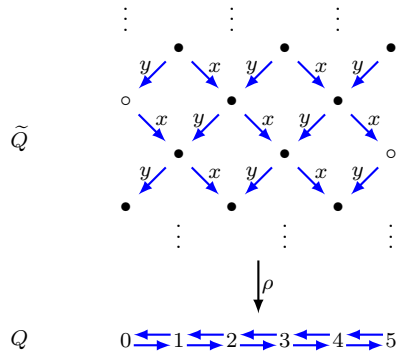
\begin{figure}\centering
\begin{tikzpicture}[scale=0.7]
\path (-2,0.3) node {$\QGam$};
\path (-2,-3.5) node {$\QG$};
\foreach \j in {0,...,5}
  {\path (\j,-3.5) node (p\j) {$\j$};
   \path (\j,-3.6) node (q\j) {~};
   \path (\j,-3.4) node (r\j)  {~};}
\path[->,>=latex] (2.5,-2) edge[black,thick] node[black,right=-2pt] {$\rho$} (2.5,-3);
\foreach \j/\k in {0/1,1/2,2/3,3/4,4/5}
 {\path[->,>=latex] (q\j) edge[blue,thick] (q\k);
  \path[->,>=latex] (r\k) edge[blue,thick] (r\j); }  
\path (0,1) node (b0) {$\circ$}; 
\path (5,0) node (b5) {$\circ$}; 
\foreach \v/\x/\y in
  {z3/3/2, a1/1/2, a2/2/1, a3/3/0, a4/4/1, 
   b1/1/0, b2/2/-1, b4/4/-1, a5/5/2, c0/0/-1}
  {\path (\x,\y) node (\v) {$\bullet$};}
\foreach \v/\j in {1,3,5}
  {\path (\j,-1.4) node {$\vdots$};
   \path (\j-1, 2.7) node {$\vdots$};}
\foreach \t/\h in
  {z3/a4, a1/a2, a2/a3, b0/b1, b1/b2, a3/b4, a4/b5}
  {\path[->,>=latex] (\t) edge[blue,thick] node[black,above right=-2pt] {$x$} (\h);}
\foreach \t/\h in
  {z3/a2, a4/a3, a5/a4, b5/b4, a3/b2, a2/b1, a1/b0, b1/c0}
  {\path[->,>=latex] (\t) edge[blue,thick] node[black,above left =-3pt] {$y$} (\h);}
\end{tikzpicture}
\caption{A cylindrical covering of a circular quiver (unwrapped)}
\label{fig:covering}
\end{figure}

If we impose the relations $xy=yx$ and $x^k=y^{n-k}$ on the path algebra of $\QGam$,
then the algebra we obtain is the incidence algebra of a poset 
$(\Gamma\dual, \leq)$.
This is closely related to the poset $I(A)$ of \cite[(13.15)]{Sim}. 

There is a covering functor (cf. \cite[(13.16)]{Sim}), given by completion,
\begin{equation}\label{eq:compl}
  \compl\colon \CM_\Gamma ( \Rb ) \to \CM_G( \Rh )
  \colon \bigoplus_{i\in\Gamma\dual} M_i \mapsto  \prod_{i\in\Gamma\dual} M_i
\end{equation}
which is invariant under the $\ZZ$-shift on $\CM_\Gamma ( \Rb )$,
as the product is just considered to be $G\dual$-graded, 
via $\rho\colon\Gamma\dual\to G\dual$. 
We may also realise this covering functor as the composite
of the forgetful functor $\CM_\Gamma(\Rb)\to \CM_G(\Rb)$ with the
comparison functor $c\colon \CM(\Ab) \to \CM(\Ah)$ of \eqref{eq:comparison}.

Modules $M$ in $\CM_\Gamma ( \Rb )$ are finitely-generated and torsion-free
and hence correspond precisely to finitely-generated, complete poset 
representations (in the sense of \cite[\S13.2]{Sim}), that is, 
representations in which all arrows are inclusions and
\begin{equation}\label{eq:Mlim}
M_i= \begin{cases}
     0 & \text{for $i\ll 0$,}\\
     \Mlim & \text{for $i\gg 0$,}
     \end{cases}
\end{equation} 
where $\Mlim$ is a vector space of dimension $\rk M$.
Thus every module $M$ in $\CM_\Gamma ( \Rb )$
may be identified with a finite subspace configuration
in $\Mlim$, pulled-back from a finite quotient poset of $\Gamma\dual$,
determined by its dimension vector
\[
 \alpha\colon \Gamma\dual \to \NN\colon i\mapsto \dim M_i.
\]

\begin{remark}\label{rem:Qgrade}
We may also consider $\Rb$, and hence $\Ab=\Rb \tgr G$, to be 
$\QQ$-graded rings with $\deg x=k/n$ and $\deg y = (n-k)/n$. 
We choose this particular grading so that the centre of $\Ab$ is $\k[t]$ with $\deg t=1$,
but one could rescale to make all $\ZZ$-graded if preferred.
There is then a natural completion map $\CM^\QQ(\Ab)\to \CM(\Ah)$,
invariant under $\QQ$-shift, which is essentially equivalent to \eqref{eq:compl}
in the following sense.
By regarding $\QQ$-graded $\Ab$-modules as $G\dual \times \QQ$-graded
$\Rb$-modules, they may also be considered as representations of an infinite 
quiver with vertex set $G\dual \times \QQ$.
The identity component of this quiver is precisely the 
McKay quiver $\QGam$ and every other component is a $\QQ$ shift of this.
Thus the category $\CM_\Gamma ( \Rb )$ 
may be considered as a subcategory of $\CM^\QQ(\Ab)$,
which is essentially equivalent to the whole category in that every indecomposable
in $\CM^\QQ(\Ab)$ is a $\QQ$-shift of an indecomposable in $\CM_\Gamma ( \Rb )$.
\end{remark}

\begin{lemma}\label{lem:gradable}
Every rigid module $M$ in $\CM_G( \Rh )$ is the completion of
a module $\Mtil$ in $\CM_\Gamma ( \Rb )$, i.e. $M\isom \compl(\Mtil)$.
If further $M$ is indecomposable, then $\Mtil$ is unique up to grade shift.
\end{lemma}

\begin{proof}
The fact that any rigid object in $\md_G(\Rh)= \md(\Ah)$ 
lifts to $\md^\QQ(\Ab)$ follows by a minor generalisation
(to the case when $A_0$ is semisimple) of an argument of
Keller-Murfet-Van den Bergh \cite[Prop.~6.1]{KMV}.
As observed in Remark~\ref{rem:Qgrade}, 
by shifting indecomposable summands, we may move the lift into $\md_\Gamma ( \Rb )$.
If $M$ is free over the centre $\Zh=\k[[t]]$, 
then $\Mtil$ will be free over $\Zb=\k[t]$,
so the lift of a CM-module is a CM-module.

A similarly minor modification of \cite[Prop~9]{AR89}  
says that $\Mtil$ is unique up to grade shift when $M$ is indecomposable.
One may also extend Keller-Murfet-Van den Bergh's argument to obtain this by noting
that, since two connections differ by an endomorphism, when $M$ has 
local endomorphism ring, the eigenvalues of the two connections,
and thus the gradings on $\Mtil$, will differ by a shift. 
\end{proof}

\begin{remark}\label{rem:rigid}
A modification of \cite[Prop~8(b)]{AR89} says that the completion 
of any indecomposable module is indecomposable.
However, while the lift $\Mtil$ of a rigid module $M$ must be rigid,
the converse is not necessarily true.
This is because $\Ext_\Gamma^1(\Mtil,\Mtil)$ is just given by the degree 0 
part of $\Ext_G^1(M,M)$.
\end{remark}

We can use the graded lift of a rigid indecomposable module 
in $\CM(\Ah)=\CM_G(\Rh)$ to give a compact notation
for it, as we will now explain.
Observe first that the dimension vector $\alpha$ of any module in $\CM_\Gamma ( \Rb )$ is
a non-decreasing function on the poset $\Gamma\dual$ and furthermore, the sets
$S_d = \{i\in \Gamma\dual : \alpha_i \geq d \}$
will be order ideals and may be described by the \emph{contours} which separate $S_d$ from its
complement. Each such contour is a sequence of up or down steps 
(as in Figure~\ref{fig:contours}) cutting arrows 
of type $x$ or $y$, respectively,
which encircles the cylinder,
so that one arrow of each index $1,\ldots, n$ is cut.
Thus the shape of a contour (which determines it up to shift) may be described 
by the $k$-subset $I\sub\intvl{1}{n}$ that gives the down-steps,
i.e the indices on the $y$-arrows that are cut by the contour.

Figure~\ref{fig:contours} shows an example of a dimension vector with its contours (in the middle),
together with a list of the shapes of those contours (on the left) and the
finite quotient poset (on the right) from which any representation of this type must pull-back.

\begin{definition}\label{def:profile}
The (ordered) collection of contour shapes for a module or dimension vector 
will be called its \emph{profile}.
\end{definition}

\begin{figure}\centering
\begin{tikzpicture}[scale=0.4]
\newcommand{\offset}{(-4,3)}
\draw \offset++(0,1) node {\sm 137};
\draw \offset++(-0.75,0.5)--++(1.5,0);
\draw \offset++(0,0) node {\sm 125};
\draw \offset++(-0.75,-0.5)--++(1.5,0);
\draw \offset++(0,-1) node {\sm 124};
\draw \offset++(-0.75,-1.5)--++(1.5,0);
\draw \offset++(0,-2) node {\sm 238};
\foreach \x/\y/\dim in
  {3/4/1, 6/5/1,7/6/1, 1/2/3, 4/1/3, 9/4/3,
   0/3/2, 1/4/2, 2/3/2, 3/2/2, 4/3/2, 5/2/2, 5/4/2, 6/3/2, 7/4/2, 8/5/2, 9/6/2,
   0/1/4, 1/0/4, 2/1/4, 3/0/4, 4/-1/4, 5/0/4, 6/1/4, 7/2/4, 8/3/4, 9/2/4}
  {\path[black] (\x,\y) node {\sm \dim};} 
\draw [blue,thick] (0,4.1)
 --++(1,1)--++(1,-1)--++(1,1)--++(1,-1)--++(1,1)--++(1,1)--++(1,1)--++(1,-1) 
 --++(1,1);
\draw [blue,thick] (0,3.9)
 --++(1,1)--++(1,-1)--++(1,-1)--++(1,1)--++(1,1)--++(1,-1)--++(1,1)--++(1,1) 
 --++(1,1);
\draw [blue,thick] (0,2.1)
 --++(1,1)--++(1,-1)--++(1,-1)--++(1,1)--++(1,-1)--++(1,1)--++(1,1)--++(1,1) 
 --++(1,1);
\draw [blue,thick] (0,1.9)
 --++(1,-1)--++(1,1)--++(1,-1)--++(1,-1)--++(1,1)--++(1,1)--++(1,1)--++(1,1) 
 --++(1,-1);
\draw (2,-0.5) node {$\vdots$};
\draw (6,-0.5) node {$\vdots$};
\draw (8, 1.5) node {$\vdots$};
\draw[dashed] (0.5,4.75)--(0.5,-1);
\draw[dashed] (8.5,6.75)--(8.5,-1);
\foreach \j in {0,...,8}
 { \path (0.5+\j,-2) node [black] {$\j$};}
 \renewcommand{\offset}{(15,2)} 
 \path \offset++(0,-1.5) node (a) {\sm 4}; 
 \path \offset++(-1,0) node (b1){\sm 3};
 \path \offset++(1,0)  node (b2) {\sm 3};
 \path \offset++(0,1.5)  node (c) {\sm 2};
 \draw \offset++(-1,3) node (d1) {\sm 1};
 \draw \offset++(1,3) node (d2) {\sm 1};
 \draw (a)--(b1);  \draw (a)--(b2); 
 \draw (b1)--(c);  \draw (b2)--(c);
 \draw (c)--(d1);  \draw (c)--(d2);
\end{tikzpicture} 
\caption{Profile, contours and poset for some dimension vector.}
\label{fig:contours}
\end{figure}
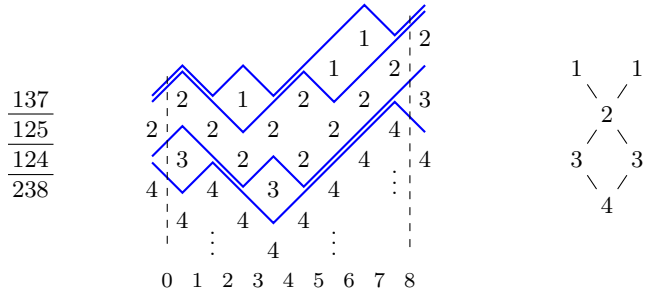

Note that the contours of an indecomposable module must be `close-packed' in the sense
that there is no (unoriented) path in $\QGam$ that goes around the cylinder between the contours.
If there were, then the associated quotient poset would have
a pinch-point and an indecomposable module
must be constant above or below this point.
In the example in Figure~\ref{fig:contours}, the contours below and above 2 are not close-packed and the quotient
poset has a pinch point at 2. Hence there will be the following direct sum decomposition.

\begin{equation*}
\begin{tikzpicture}[xscale=0.35,yscale=0.25]
\newcommand{\offset}{(0,2)} 
 \path \offset++(0,-1.5) node (a) {$4$}; 
 \path \offset++(-1,0) node (b1){$3$};
 \path \offset++(1,0)  node (b2) {$3$};
 \path \offset++(0,1.5)  node (c) {$2$};
 \draw \offset++(-1,3) node (d1) {$1$};
 \draw \offset++(1,3) node (d2) {$1$};
 \draw \offset++(2.5,1) node {$=$};
 \renewcommand{\offset}{(5,2)} 
 \path \offset++(0,-1.5) node (a) {$2$}; 
 \path \offset++(-1,0) node (b1){$2$};
 \path \offset++(1,0)  node (b2) {$2$};
 \path \offset++(0,1.5)  node (c) {$2$};
 \draw \offset++(-1,3) node (d1) {$1$};
 \draw \offset++(1,3) node (d2) {$1$};
 \draw \offset++(2.5,1) node {$\oplus$};
 \renewcommand{\offset}{(10,2)} 
 \path \offset++(0,-1.5) node (a) {$2$}; 
 \path \offset++(-1,0) node (b1){$1$};
 \path \offset++(1,0)  node (b2) {$1$};
 \path \offset++(0,1.5)  node (c) {$0$};
 \draw \offset++(-1,3) node (d1) {$0$};
 \draw \offset++(1,3) node (d2) {$0$};
\end{tikzpicture} 
\end{equation*}

Thus, by Lemma~\ref{lem:gradable}, we can assign to any rigid indecomposable module 
$M$ in $\CM_G(\Rh)$ the profile of a graded lift $\Mtil$ in $\CM_\Gamma(\Rb)$. 
This profile is an invariant of $M$ and determines the dimension vector of $\Mtil$ up to a shift.

\begin{example}\label{ex:profiles}
Figure~\ref{fig:3contours}
shows a collection of profiles of certain rigid indecomposable modules
in $\CM_G(\Rh)$, 
in the case $n=8$, $k=3$, together their lifted graded dimension vectors 
(with 0's omitted) with their contours and 
the associated type of subspace configurations.

In the third subspace configuration, the first `1' must be the intersection of the two `2's
that it is contained in, so may effectively be omitted. Thus all three types of subspace configuration
correspond to roots of $E_6$, and so there is a unique rigid indecomposable configuration of this type
and hence a unique (up to shift) rigid indecomposable graded module 
with the corresponding dimension vector and profile.
It can be independently checked that the completions of these graded modules
are also rigid (cf. Remark~\ref{rem:rigid}).
\end{example}

\begin{figure}
\begin{center}
\begin{tikzpicture}[scale=0.4]
\newcommand{\offset}{(-4,2)}
\draw \offset++(0,1) node {\sm 368};
\draw \offset++(-0.75,0.5)--++(1.5,0);
\draw \offset++(0,0) node {\sm 258};
\draw \offset++(-0.75,-0.5)--++(1.5,0);
\draw \offset++(0,-1) node {\sm 147};
\foreach \v/\x/\y/\dim in
  {01/0/1/3, 12/1/2/3, 21/2/1/3, 32/3/2/3, 43/4/3/3, 52/5/2/3,
   63/6/3/3, 74/7/4/3, 83/8/3/3, 94/9/4/3, 03/0/3/2, 85/8/5/2,
   23/2/3/2, 34/3/4/1, 54/5/4/2, 65/6/5/1}
  {\path[black] (\x,\y) node (x\v) {\sm \dim};} 
\draw [blue,thick] (0,4.15)
 --++(1,-1)--++(1,1)--++(1,1)--++(1,-1)--++(1,1)--++(1,1)--++(1,-1)--++(1,1) 
 --++(1,-1);
\draw [blue,thick] (0,3.95)
 --++(1,-1)--++(1,1)--++(1,-1)--++(1,1)--++(1,1)--++(1,-1)--++(1,1)--++(1,1) 
 --++(1,-1);
\draw [blue,thick] (0,1.75)
 --++(1,1)--++(1,-1)--++(1,1)--++(1,1)--++(1,-1)--++(1,1)--++(1,1)--++(1,-1) 
 --++(1,1);
\draw (1,0.5) node {$\vdots$};
\draw (4,1.5) node {$\vdots$};
\draw (7,2.5) node {$\vdots$};
\draw[dashed] (0.5,4.25)--(0.5,0);
\draw[dashed] (8.5,6.25)--(8.5,0);
\foreach \j in {0,...,8}
 { \path (0.5+\j,-1) node [black] {$\j$};}
 \renewcommand{\offset}{(15,2)} 
 \path \offset++(0,-2) node (a) {\sm 3}; 
 \path \offset++(-2,0) node (b1){\sm 2};
 \path \offset++(0,0)  node (b2) {\sm 2};
 \path \offset++(2,0)  node (b3) {\sm 2};
 \draw \offset++(-2,2) node (c1) {\sm 1};
 \draw \offset++(0,2) node (c2) {\sm 1};
 \draw (a)--(b1);  \draw (a)--(b2);  \draw (a)--(b3);
 \draw (b1)--(c1);  \draw (b2)--(c2); 
\end{tikzpicture} 
\end{center}

\begin{center}
\begin{tikzpicture}[scale=0.4]
\newcommand{\offset}{(-4,2)}
\draw \offset++(0,1) node {\sm 147};
\draw \offset++(-0.75,0.5)--++(1.5,0);
\draw \offset++(0,0) node {\sm 368};
\draw \offset++(-0.75,-0.5)--++(1.5,0);
\draw \offset++(0,-1) node {\sm 258};
\foreach \v/\x/\y/\dim in
  {01/0/1/3, 10/1/0/3, 21/2/1/3, 30/3/0/3, 41/4/1/3, 52/5/2/3,
   61/6/1/3, 72/7/2/3, 83/8/3/3, 92/9/2/3, 
   12/1/2/1, 32/3/2/2, 43/4/3/1, 63/6/3/2, 74/7/4/1,94/9/4/1}
  {\path[black] (\x,\y) node (x\v) {\sm \dim};} 
\draw [blue,thick] (0,2.15)
 --++(1,1)--++(1,-1)--++(1,1)--++(1,1)--++(1,-1)--++(1,1)--++(1,1)--++(1,-1) 
 --++(1,1);
\draw [blue,thick] (0,1.95)
 --++(1,-1)--++(1,1)--++(1,1)--++(1,-1)--++(1,1)--++(1,1)--++(1,-1)--++(1,1) 
 --++(1,-1);
\draw [blue,thick] (0,1.75)
 --++(1,-1)--++(1,1)--++(1,-1)--++(1,1)--++(1,1)--++(1,-1)--++(1,1)--++(1,1) 
 --++(1,-1);
\draw (2,-0.5) node {$\vdots$};
\draw (5,0.5) node {$\vdots$};
\draw (8,1.5) node {$\vdots$};
\draw[dashed] (0.5,3.25)--(0.5,-1);
\draw[dashed] (8.5,5.25)--(8.5,-1);
\foreach \j in {0,...,8}
 { \path (0.5+\j,-2) node [black] {$\j$};}
 \renewcommand{\offset}{(15,2)} 
 \path \offset++(0,-2) node (a) {\sm 3}; 
 \path \offset++(-2,0) node (b1){\sm 1};
 \path \offset++(0,0)  node (b2) {\sm 2};
 \path \offset++(2,0)  node (b3) {\sm 2};
 \draw \offset++(0,2) node (c2) {\sm1};
 \draw \offset++(2,2) node (c3) {\sm 1};
 \draw (a)--(b1);  \draw (a)--(b2);  \draw (a)--(b3);
 \draw (b3)--(c3);  \draw (b2)--(c2); 
\end{tikzpicture} 
\end{center}

\begin{center}
\begin{tikzpicture}[scale=0.4]
\newcommand{\offset}{(-4,2)}
\draw \offset++(0,1) node {\sm 258};
\draw \offset++(-0.75,0.5)--++(1.5,0);
\draw \offset++(0,0) node {\sm 147};
\draw \offset++(-0.75,-0.5)--++(1.5,0);
\draw \offset++(0,-1) node {\sm 368};
\foreach \v/\x/\y/\dim in
  {01/0/1/3, 10/1/0/3, 21/2/1/3, 32/3/2/3, 41/4/1/3, 52/5/2/3,
   63/6/3/3, 72/7/2/3, 83/8/3/3, 92/9/2/3, 
   03/0/3/1, 12/1/2/2, 23/2/3/1, 43/4/3/2, 54/5/4/1,74/7/4/2,85/8/5/1,94/9/4/2}
  {\path[black] (\x,\y) node (x\v) {\sm \dim};} 
\draw [blue,thick] (0,4.15)
 --++(1,-1)--++(1,1)--++(1,-1)--++(1,1)--++(1,1)--++(1,-1)--++(1,1)--++(1,1) 
 --++(1,-1);
\draw [blue,thick] (0,1.95)
 --++(1,1)--++(1,-1)--++(1,1)--++(1,1)--++(1,-1)--++(1,1)--++(1,1)--++(1,-1) 
 --++(1,1);
\draw [blue,thick] (0,1.75)
 --++(1,-1)--++(1,1)--++(1,1)--++(1,-1)--++(1,1)--++(1,1)--++(1,-1)--++(1,1) 
 --++(1,-1);
\draw (3,0.5) node {$\vdots$};
\draw (6,1.5) node {$\vdots$};
\draw (8,1.5) node {$\vdots$};
\draw[dashed] (0.5,4.25)--(0.5,0);
\draw[dashed] (8.5,6.25)--(8.5,0);
\foreach \j in {0,...,8}
 { \path (0.5+\j,-1) node [black] {$\j$};}
 \renewcommand{\offset}{(15,2)} 
 \path \offset++(0,-2) node (a) {\sm 3}; 
 \path \offset++(-2,0) node (b1){\sm 2};
 \path \offset++(0,0)  node (b2) {\sm 2};
 \path \offset++(2,0)  node (b3) {\sm 2};
 \draw \offset++(-2,2) node (c1) {\sm 1};
 \draw \offset++(0,2) node (c2) {\sm 1};
 \draw \offset++(2,2) node (c3) {\sm 1};
 \draw (a)--(b1);  \draw (a)--(b2);  \draw (a)--(b3);
 \draw (b3)--(c3);  \draw (b2)--(c2);  \draw (b1)--(c1);
 \draw (c1)--(b2);
\end{tikzpicture} 
\end{center}
\caption{Profile, contours and poset for three rigid indecomposables.}
\label{fig:3contours}
\end{figure}
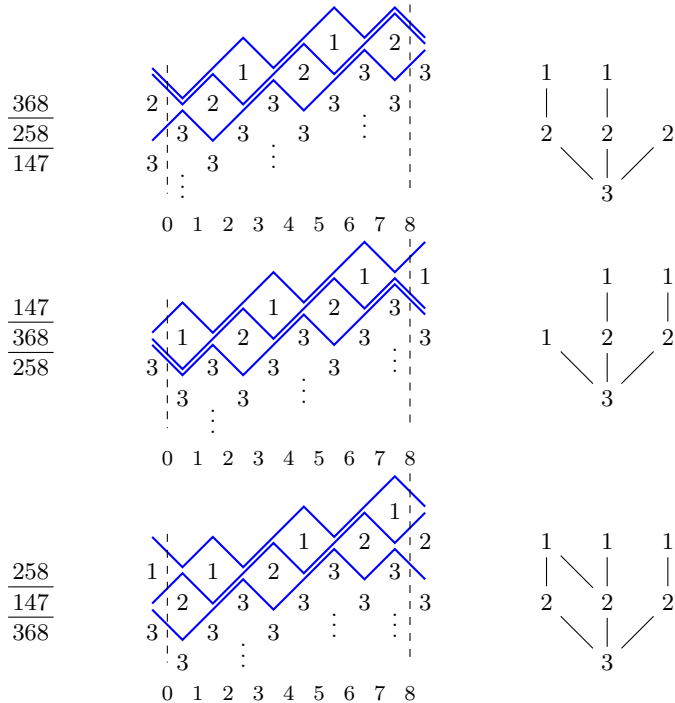

The profile of any graded module, and thus of any rigid indecomposable, 
has the following useful interpretation.

\begin{proposition}\label{prop:genfilt}
Every module $M$ in $\CM_\Gamma(\Rb)$ has a filtration with factors
which are rank 1 modules whose types are the $k$-subsets in its profile.
\end{proposition}

\begin{proof}
We proceed by induction on $r=\rk M$. The case $r=1$ is straightforward,
since the dimension vector of a rank 1 module has one contour, whose shape coincides with
its type, i.e. the label $I$ in Definition~\ref{def:rkone}.

Let $\Mlim$ be the limiting vector space for $M$, as in \eqref{eq:Mlim}.
Note that any subspace $V_*\leq \Mlim$ determines a submodule $V$ of $M$,
by setting $V_i=V_*\cap M_i$, and that this submodule is `strict' in the sense
that $M/V$ is torsion-free and so still in $\CM_\Gamma(\Rb)$,
i.e. the inclusion $V\hookrightarrow M$ is a strict monomorphism in the exact structure
on $\CM_\Gamma(\Rb)$.
Indeed, the strict submodules are precisely those of this form.

If we choose $V_*$ to be 1-dimensional and generic, in the sense that it has
trivial intersection with every proper subspace $M_i\leq \Mlim$ in the representation, 
then $V$ will be a rank 1 submodule whose type is the shape of the bottom contour 
in the profile of $M$ and thus $V$ can be the bottom term in the filtration.
Furthermore, $M/V$ will have rank $r-1$ and a profile given by the rest of the profile
of $M$ and so the induction proceeds.
\end{proof}

\begin{corollary}\label{cor:genfilt}
Every rigid indecomposable module in $\CM_G(\Rh)$ has a filtration with factors
which are rank 1 modules whose types are the $k$-subsets in its profile.
\end{corollary}

We see from the proof above, that such a profile-matching filtration,
as constructed in Proposition~\ref{prop:genfilt}, is not canonical;
indeed it arises from a generic choice of full flag in $\Mlim$ 
and so we call it a \emph{generic filtration}.

\goodbreak\section{Examples}
\label{sec:examples} 

For the rest of the paper we restrict to the case $A=\Ah$, so that we can talk about 
Auslander-Reiten (A-R) quivers and Grothendieck groups without ambiguity.

For $\grass36$, the A-R quiver of $\CM(A)$ is shown in Figure~\ref{fig:ARqG36}.
Note: the quiver is periodic, with the dotted lines identified by a straight translation;
all indecomposables are rigid and are denoted by their profiles;
arrows go left-to-right, but arrowheads and meshes are omitted.
\begin{figure}\centering
\begin{tikzpicture}[scale=1.1]
\draw[dotted] (0.5,2.2)--(0.5,-2.2);
\foreach \v/\y/\lab in
  {a3/0/$\dfrac{135}{246}$}
  {\path (0,\y) node(\v) {\lab}; }
\foreach \v/\y/\lab in
  {a2/1/125, b2/-0.3/134, c2/-1.3/356}
  {\path (1,\y) node(\v) {\lab}; }
\foreach \t/\h in
  {a3/a2,a3/b2, a3/c2}
  {\draw (\t)--(\h);}
\foreach \v/\y/\lab in
  {a1/2/126, a3/0/135, b1/-1/234, c1/-2/456}
  {\path (2,\y) node(\v) {\lab}; }
\foreach \t/\h in
  {a1/a2,a2/a3,a3/b2, b2/b1, a3/c2, c2/c1}
  {\draw (\t)--(\h);}
\foreach \v/\y/\lab in
  {a2/1/136, b2/-0.3/235, c2/-1.3/145}
  {\path (3,\y) node(\v) {\lab}; }
\foreach \t/\h in
  {a1/a2,a2/a3,a3/b2, b2/b1, a3/c2, c2/c1}
  {\draw (\t)--(\h);}
\foreach \v/\y/\lab in
  {a3/0/$\dfrac{246}{135}$}
  {\path (4,\y) node(\v) {\lab}; }
\foreach \t/\h in
  {a3/a2, a3/b2, a3/c2}
  {\draw (\t)--(\h);}
\foreach \v/\y/\lab in
  {a2/1/245, b2/-0.3/146, c2/-1.3/236}
  {\path (5,\y) node(\v) {\lab}; }
\foreach \t/\h in
  {a2/a3, b2/a3, c2/a3}
  {\draw (\t)--(\h);}
\foreach \v/\y/\lab in
  {a1/2/345, a3/0/246, b1/-1/156, c1/-2/123}
  {\path (6,\y) node(\v) {\lab}; }
\foreach \t/\h in
  {a1/a2,a2/a3,a3/b2, b2/b1, a3/c2, c2/c1}
  {\draw (\t)--(\h);}
\foreach \v/\y/\lab in
  {a2/1/346, b2/-0.3/256, c2/-1.3/124}
  {\path (7,\y) node(\v) {\lab}; }
\foreach \t/\h in
  {a1/a2,a2/a3,a3/b2, b2/b1, a3/c2, c2/c1}
  {\draw (\t)--(\h);}
\foreach \v/\y/\lab in
  {a3/0/$\dfrac{135}{246}$}
  {\path (8,\y) node(\v) {\lab}; }
\foreach \t/\h in
  {a3/a2, a3/b2, a3/c2}
  {\draw (\t)--(\h);}
\foreach \v/\y/\lab in
  {a2/1/125, b2/-0.3/134, c2/-1.3/356}
  {\path (9,\y) node(\v) {\lab}; }
\foreach \t/\h in
  {a2/a3, b2/a3, c2/a3}
  {\draw (\t)--(\h);}
\draw[dotted] (8.5,2.2)--(8.5,-2.2);
\end{tikzpicture}
\caption{The Auslander-Reiten quiver of $\CM(A)$ for $\grass36$.}
\label{fig:ARqG36}
\end{figure}

To compute such an example, we may compute the A-R quiver of $\Sub Q_k$ and then
lift the indecomposable modules to  $\CM(A)$ using Theorem~\ref{thm:main}.
What remains is to locate the new projective $P_n$, 
i.e. the rank 1 module labelled $123$ in this case.
For this, we observe that there is an irreducible map from its radical $236 \to 123$
and then that $\tau^{-1}(236) = \Omega(236) = 124$.
We can then compute the almost split sequence
\begin{equation}\label{eq:ASS}
  236 \to 246 \oplus 123 \to 124
\end{equation}
which confirms the location of 123  in Figure~\ref{fig:ARqG36}.
Note that $\trun(124)$ is the simple module $S_k$, and it will
always be the case that the lift of $S_k$ is the codomain of 
the (unique) irreducible
map from $P_n$.

Alternatively, we may use the covering \eqref{eq:compl} to compute the A-R quiver directly,
as in \cite[Thm~13.17]{Sim} or \cite{RW}. 
Indeed, the next two examples are simliar to \cite[Examples 13.28 \& 13.18]{Sim}, 
where the same poset covers a different tiled order.

For $\grass37$, a piece of the A-R quiver  is shown in Figure~\ref{fig:ARqG37}.
The quiver may be continued to the right by adding 3 to all labels (mod 7) until
it becomes periodic under a glide reflection.
For $\grass38$ a piece of the A-R quiver, containing 2 of the 8 projectives, is shown in Figure~\ref{fig:ARqG38}.
The quiver may be continued to the right by adding 3 to all labels (mod 8) until
it becomes periodic under a translation.

\begin{figure}\centering
\begin{tikzpicture}[scale=1.1]
\foreach \ver/\pos/\lab in
  {a2/2/356, a4/0/$\dfrac{135}{247}$, , c2/-2/267}
  {\path (0,\pos) node(\ver) {\lab}; }
\foreach \ver/\pos/\lab in
  {a1/3/456, a3/1/357, b2/0/125, c3/-1/$\dfrac{136}{247}$}
  {\path (1,\pos) node(\ver) {\lab}; }
\foreach \x/\y in
  {a1/a2,a2/a3,a3/a4,b2/a4,a4/c3,c3/c2}
  {\draw (\x)--(\y);}
\foreach \ver/\pos/\lab in
  {a2/2/457, a4/0/$\dfrac{136}{257}$, c2/-2/134}
  {\path (2,\pos) node(\ver) {\lab}; }
\foreach \x/\y in
  {a1/a2,a2/a3,a3/a4,b2/a4,a4/c3,c3/c2}
  {\draw (\x)--(\y);}
\foreach \ver/\pos/\lab in
  {a3/1/$\dfrac{146}{257}$, b2/0/367, c3/-1/135,c1/-3/234}
  {\path (3,\pos) node(\ver) {\lab}; }
\foreach \x/\y in
  {a2/a3,a3/a4,b2/a4,a4/c3,c3/c2,c2/c1}
  {\draw (\x)--(\y);}
\foreach \ver/\pos/\lab in
  {a2/2/126, a4/0/$\dfrac{146}{357}$, , c2/-2/235}
  {\path (4,\pos) node(\ver) {\lab}; }
\foreach \x/\y in
  {a2/a3,a3/a4,b2/a4,a4/c3,c3/c2,c2/c1}
  {\draw (\x)--(\y);}
\foreach \ver/\pos/\lab in
  {a1/3/127, a3/1/136, b2/0/145, c3/-1/$\dfrac{246}{357}$}
  {\path (5,\pos) node(\ver) {\lab}; }
\foreach \x/\y in
  {a1/a2,a2/a3,a3/a4,b2/a4,a4/c3,c3/c2}
  {\draw (\x)--(\y);}
\foreach \ver/\pos/\lab in
  {a2/2/137, a4/0/$\dfrac{246}{135}$, c2/-2/467}
  {\path (6,\pos) node(\ver) {\lab}; }
\foreach \x/\y in
  {a1/a2,a2/a3,a3/a4,b2/a4,a4/c3,c3/c2}
  {\draw (\x)--(\y);}
\foreach \ver/\pos/\lab in
  {a3/1/$\dfrac{247}{135}$, b2/0/236, c3/-1/146, c1/-3/567}
  {\path (7,\pos) node(\ver) {\lab}; }
\foreach \x/\y in
  {a2/a3,a3/a4,b2/a4,a4/c3,c3/c2,c2/c1}
  {\draw (\x)--(\y);}
\foreach \ver/\pos/\lab in
  {a2/2/245, a4/0/$\dfrac{247}{136}$, , c2/-2/156}
  {\path (8,\pos) node(\ver) {\lab}; }
\foreach \x/\y in
  {a2/a3,a3/a4,b2/a4,a4/c3,c3/c2,c2/c1}
  {\draw (\x)--(\y);}
\foreach \ver/\pos/\lab in
  {a1/3/345, a3/1/246, b2/0/147, c3/-1/$\dfrac{257}{136}$}
  {\path (9,\pos) node(\ver) {\lab}; }
\foreach \x/\y in
  {a1/a2,a2/a3,a3/a4,b2/a4,a4/c3,c3/c2}
  {\draw (\x)--(\y);}
\end{tikzpicture}
\caption{Part of the Auslander-Reiten quiver of $\CM(A)$ for $\grass37$.}
\label{fig:ARqG37}
\end{figure}
 
\begin{figure}\centering
\begin{tikzpicture}[scale=1.1]
\foreach \ver/\pos/\lab in
  {a2/4/138, a4/2/$\dfrac{247}{135}$, a6/0/$\thfrac{257}{146}{368}$, c2/-2/478}
  {\path (0,\pos) node(\ver) {\lab};}
\foreach \ver/\pos/\lab in
  {a3/3/$\dfrac{248}{135}$, a5/1/$\dfrac{247}{136}$, b3/0/$\dfrac{257}{368}$, c4/-1/$\dfrac{157}{468}$}
  {\path (1,\pos) node(\ver) {\lab};}
\foreach \x/\y in
  {a2/a3,a3/a4,a4/a5,a5/a6,b3/a6,a6/c4,c4/c2}
  {\draw (\x)--(\y);}
\foreach \ver/\pos/\lab in
  {a2/4/245, a4/2/$\dfrac{248}{136}$, a6/0/$\thfrac{257}{147}{368}$, c2/-2/156}
  {\path (2,\pos) node(\ver) {\lab};}
\foreach \x/\y in
  {a2/a3,a3/a4,a4/a5,a5/a6,b3/a6,a6/c4,c4/c2}
  {\draw (\x)--(\y);}
\foreach \ver/\pos/\lab in
  {a1/5/345, a3/3/246, a5/1/$\thfrac{258}{147}{368}$, b3/0/147, c4/-1/$\dfrac{257}{136}$}
  {\path (3,\pos) node(\ver) {\lab};}
\foreach \x/\y in
  {a1/a2,a2/a3,a3/a4,a4/a5,a5/a6,b3/a6,a6/c4,c4/c2}
  {\draw (\x)--(\y);}
\foreach \ver/\pos/\lab in
  {a2/4/346, a4/2/$\dfrac{257}{468}$, a6/0/$\thfrac{258}{147}{136}$, c2/-2/237}
  {\path (4,\pos) node(\ver) {\lab};}
\foreach \x/\y in
  {a1/a2,a2/a3,a3/a4,a4/a5,a5/a6,b3/a6,a6/c4,c4/c2}
  {\draw (\x)--(\y);}
\foreach \ver/\pos/\lab in
  {a3/3/$\dfrac{357}{468}$, a5/1/$\dfrac{257}{146}$, b3/0/$\dfrac{258}{136}$, c4/-1/$\dfrac{248}{137}$}
  {\path (5,\pos) node(\ver) {\lab};}
\foreach \x/\y in
  {a2/a3,a3/a4,a4/a5,a5/a6,b3/a6,a6/c4,c4/c2}
  {\draw (\x)--(\y);}
\foreach \ver/\pos/\lab in
  {a2/4/578, a4/2/$\dfrac{357}{146}$, a6/0/$\thfrac{258}{247}{136}$, c2/-2/148}
  {\path (6,\pos) node(\ver) {\lab};}
\foreach \x/\y in
  {a2/a3,a3/a4,a4/a5,a5/a6,b3/a6,a6/c4,c4/c2}
  {\draw (\x)--(\y);}
\foreach \ver/\pos/\lab in
  {a1/5/678, a3/3/157, a5/1/$\thfrac{358}{247}{136}$, b3/0/247, c4/-1/$\dfrac{258}{146}$}
  {\path (7,\pos) node(\ver) {\lab};}
\foreach \x/\y in
  {a1/a2,a2/a3,a3/a4,a4/a5,a5/a6,b3/a6,a6/c4,c4/c2}
  {\draw (\x)--(\y);}
\foreach \ver/\pos/\lab in
  {a2/4/167, a4/2/$\dfrac{258}{137}$, a6/0/$\thfrac{358}{247}{146}$, c2/-2/256}
  {\path (8,\pos) node(\ver) {\lab};}
\foreach \x/\y in
  {a1/a2,a2/a3,a3/a4,a4/a5,a5/a6,b3/a6,a6/c4,c4/c2}
  {\draw (\x)--(\y);}
\foreach \ver/\pos/\lab in
  {a3/3/$\dfrac{268}{137}$, a5/1/$\dfrac{258}{147}$, b3/0/$\dfrac{358}{146}$, c4/-1/$\dfrac{357}{246}$}
  {\path (9,\pos) node(\ver) {\lab};}
\foreach \x/\y in
  {a2/a3,a3/a4,a4/a5,a5/a6,b3/a6,a6/c4,c4/c2}
  {\draw (\x)--(\y);}
\end{tikzpicture}
\caption{Part of the Auslander-Reiten quiver of $\CM(A)$ for $\grass38$.}
\label{fig:ARqG38}
\end{figure}

\begin{remark}\label{rem:cycling}
Notice that, in Figure~\ref{fig:ARqG38}, every label occurs in each rank 3 module, with one label occurring twice.
For example, by suitably shifting the labels on the three rank 3 modules appearing in the first four columns,
we see that the three modules from Figure~\ref{fig:3contours} are precisely those with the label `8' occurring twice:
\begin{equation}\label{eq:two8s}
\thfrac{147}{368}{258} 
\qquad
\thfrac{368}{258}{147} 
\qquad
\thfrac{258}{147}{368}
\end{equation}
It may be confirmed that all rank 2 and rank 3 modules above occur in
pairs or triples, respectively, whose profiles contain the same contour shapes,
but with the layers cyclically reordered.

Thus these modules have generic filtrations with the same rank 1 factors
and hence must have the same class in the Grothendieck group of $\CM(A)$.
This empirical phenomenon of \emph{cyclic reordering of factors} 
provides a natural explanation of the occurrence of 
multiple non-isomorphic indecomposables in the same class,
but we don't currently understand why such reordering is \emph{a priori} possible.

Note that this cyclic reordering of factors is quite different from the symmetry
of $\CM(A)$ that arises from the automorphism of $A$ corresponding to the
cyclic symmetry of the double quiver $Q(C)$ in Figure~\ref{fig:mckayQ}.
\end{remark}

\goodbreak\section{The Grothendieck group}
\label{sec:grotgrp}  

\newcommand{\x}{\mathbf{x}}

We now study the Grothendieck group of $\CM(A)$ and show how it may be identified with
the root lattice $\rootlat(J_{k,n})$ and
with the sublattice $\ZZ^n(k)\subseteq \ZZ^n$ spanned by the $\GL_n(\k)$ weights of
the homogeneous functions in $\GCA{k}{n}$.
Recall that these two lattices were already identified with each other in Section~\ref{sec:Jkn}.

First observe that the short exact sequence \eqref{eq:ses.trunM},
together with the fact (Proposition~\ref{prp:tosubQk}) that both $\trun M$ and
$\rk(M)=\dim V$ are additive on short exact sequences of CM-modules,
provides us with maps 
\begin{equation}\label{eq:Grot}
\begin{aligned}
 \Grot(\CM(A)) &\to \Grot(B) \oplus\ZZ \to \Grot(A) \\
  [M]  &\mapsto  [\trun M] + \rk(M) \mapsto [\trun M] + \rk(M)[P_n]
\end{aligned}
\end{equation}
where $\Grot(B)$ and $\Grot(A)$ are the Grothendieck groups
of finitely generated modules, which in the case of $B$ is the 
same as finite dimensional modules.
Thus $\Grot(B)$ has a basis 
$[S_1],\dots,[S_{n-1}]$ given by its simple modules.
Note also that $\Grot(B)=\Grot(\Pi)$
(cf. Remark~\ref{rem:subQk}).
 
However, a standard argument (cf. \cite[Lemma~13.2]{Yo}) shows that
the map $\Grot(\CM(A)) \to \Grot(A)$ is an isomorphism,
whose inverse is $[M]\mapsto [P]-[C]$ where $P\to M$ is any 
projective cover, whose kernel $C$ is Cohen-Macaulay
in our case.

From considering just the rank 1 modules, we can see that
the first map in \eqref{eq:Grot} is surjective and thus deduce that 
\begin{equation*}
  \Grot(\CM(A)) \isom \Grot(A) \isom \Grot(B)\oplus\ZZ
\end{equation*}
and thus that $\Grot(A)$ has a basis $[S_1],\dots,[S_{n-1}],[P_n]$.
We will then make the identification 
\begin{equation}\label{eq:ident}
  \Grot(A)\isom \ZZ^n(k)\isom \rootlat(J_{k,n})
\end{equation}
by identifying this basis with the basis $\alpha_1,\ldots,\alpha_{n-1},\beta_{\cycint{n}}$,
as described in Section~\ref{sec:Jkn}. 
By Proposition~\ref{prp:tosubQk}, the $[P_n]$ coefficient of $[M]$ in this basis is $\rk M$.
In particular, if $\Ntil\in\CM(A)$ is the minimal lift of $N\in\Sub Q_k$ provided by Theorem~\ref{thm:main},
then $\rk\Ntil=\dim\soc N$ and $[\Ntil]$ is identified with the 
enhanced dimension vector of $N$, as in Observation~\ref{obs:subQk}.

Note that, under the identification \eqref{eq:ident}, we also have
$[S_n]=e_1-e_n$, which we could call $\alpha_n$ by comparison with \eqref{eq:alpha}. 
This holds because $\sum_{i=1}^n [S_i] = 0$, 
as this is the class of $M/tM$ for any rank 1 CM-module $M$ and $M \isom tM$.

\begin{remark}\label{rem:class}
\New{Following Remark~\ref{rem:young},}
for any rank 1 CM-module $\rkone{I}$ (Def.~\ref{def:rkone}),
we may see that $\trun\rkone{I}=N_I$, so $[\trun\rkone{I}]= \beta_I - \beta_{\cycint{n}}$ and hence,
by \eqref{eq:Grot},
\begin{equation}\label{eq:Lclass}
  [\rkone{I}] = \beta_I.
\end{equation}
Thus the class of any $M\in\CM(A)$ is determined just by adding up
the labels that appear in its profile, e.g.
\[
  \left[ \dfrac{247}{135} \right] = (e_2+e_4+e_7) + (e_1+e_3+e_5) .
\]
\end{remark}

Although the observations in Section~\ref{sec:Jkn} give a strong indication that
the root lattice $\rootlat(J_{k,n})$ plays an important role in understanding $\CM(A)$,
we do not currently have any interpretation (homological or otherwise) of the quadratic form $q$.
We do however now have the following refinement of those observations.

\begin{observation}
When $\CM(A)$ has finite type (so all indecomposable modules are rigid), the classes in $\Grot(A)$ 
of the indecomposable rank $d$ modules are precisely the roots of degree $d$.
Furthermore, if $M$ is an indecomposable rank $d$ CM-module
whose generic filtration has rank 1 factors
$M_1 \mid M_2 \mid \cdots \mid M_d$,
then there is another indecomposable CM-module $M'$ of rank $d$ and
generic filtration with rank 1 factors
$M_d \mid M_1 \mid \cdots \mid M_{d-1}$,
which thus has the same class in $\Grot(A)$.
This cyclic reordering of generic factors produces $d$ non-isomorphic 
indecomposable CM-modules of rank $d$ and these are precisely
all such modules of class $[M]$.
\end{observation}

It is tempting to conjecture that the above remains true in infinite type
if we replace ``indecomposable'' by ``rigid indecomposable'' and ``root'' by ``real root''.
Unfortunately, this is not the case.
For example, when $n=8$ and $k=4$, Figure~\ref{fig:counterexample} shows the 
profile, contours and poset for the graded dimension vector of
(the lift of) a rigid indecomposable module of class
$e_1+\cdots +e_8$, which has $q=0$.
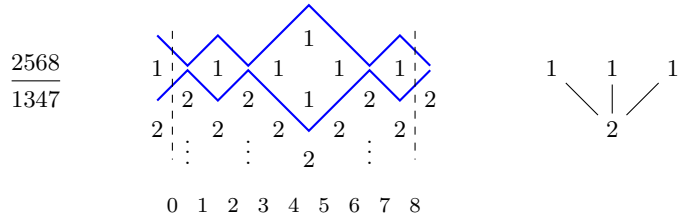
\begin{figure}\centering
\begin{tikzpicture}[scale=0.4]
\newcommand{\offset}{(-5,1)}
\draw \offset++(0,1.2) node {\sm 2568};
\draw \offset++(-0.8,0.6)--++(1.6,0);
\draw \offset++(0,0) node {\sm 1347};
\foreach \x/\y/\dim in
  {-1/0/2, 0/1/2, 1/0/2, 2/1/2, 3/0/2, 4/-1/2, 5/0/2, 6/1/2, 7/0/2, 8/1/2, 
   -1/2/1, 1/2/1, 3/2/1, 4/3/1, 4/1/1, 5/2/1, 7/2/1}
  {\path[black] (\x,\y) node {\sm \dim};} 
\draw [blue,thick] (-1,3.1)
 --++(1,-1)--++(1,1)--++(1,-1)--++(1,1)--++(1,1)--++(1,-1)--++(1,-1)--++(1,1)--++(1,-1); 
\draw [blue,thick] (-1,0.95)
 --++(1,1)--++(1,-1)--++(1,1)--++(1,-1)--++(1,-1)--++(1,1)--++(1,1)--++(1,-1)--++(1,1); 
\draw (0,-0.5) node {$\vdots$};
\draw (2,-0.5) node {$\vdots$};
\draw (6,-0.5) node {$\vdots$};
\draw[dashed] (-0.5,3.25)--(-0.5,-1);
\draw[dashed] (7.5,3.25)--(7.5,-1);
\foreach \j in {0,...,8}
 { \path (-0.5+\j,-2.5) node [black] {$\j$};}
 \renewcommand{\offset}{(14,2)} 
 \path \offset++(0,-2) node (a) {\sm 2}; 
 \path \offset++(-2,0) node (b1){\sm1};
 \path \offset++(0,0)  node (b2) {\sm 1};
 \path \offset++(2,0)  node (b3) {\sm 1};
 \draw (a)--(b1);  \draw (a)--(b2);  \draw (a)--(b3);
\end{tikzpicture} 
\caption{A rigid indecomposable with $q=0$.}
\label{fig:counterexample}
\end{figure}

For this module, the generic factors can not be reordered to obtain another rigid indecomposable.
However, the empirical evidence is at least consistent with the possibility that the classes
of all rigid indecomposable modules are roots and, in particular, have $q\leq 2$.
Furthermore, in all known examples where the class is a real root, the generic factors
can be cyclically reordered.

\goodbreak\section{Categorification}
\label{sec:conc}  

We conclude by explaining more precisely how $\CM(A)$ categorifies the cluster algebra $\GCA{k}{n}$. 
As usual, we rely heavily on Geiss-Leclerc-Schroer's inhomogeneous categorification of
$\k[\N]$ by $\Sub Q_k$.
 
As a special case of a more general result \cite[Theorem~3]{GLS06} 
(see also \cite[Prop~9.1]{GLS08})
we know that, for the dominant weight $\lambda=d\omega$ as in \eqref{eq:repsum}, 
if we set $Q_\lambda= Q_k^{\oplus d}$, 
then the cluster characters of the submodules of $Q_\lambda$ span a submodule of 
$\k[\N]$ which can be identified with the irreducible representation
of $\SL_n(\k)$ of highest weight $\lambda$, that is, 
\begin{equation}\label{eq:phispan}
  \irrep{\lambda} \isom \langle \cluschar{N} : N \leq Q_\lambda \rangle.
\end{equation}
Indeed, as in \cite[\S2.5]{GLS08}, this identification is given by
\begin{equation}\label{eq:projection}
 \dehomog\colon \irrep{d\omega} \hookrightarrow \GCA{k}{n} \to \k[\N], 
\end{equation}
where the first map is the inclusion of the summand of degree $d$, as in \eqref{eq:repsum}, 
and the second is the quotient which sets $\Pluck{\cycint{n}} = 1$, as in \eqref{eq:kN}.
Note that $\bigl(\Pluck{\cycint{n}}\bigr)^d\in \irrep{d\omega}$ is the highest weight vector.
The image of $\dehomog$ consists of polynomials $\phi\in\k[\N]$
with $\deg\phi \leq d$; see e.g. \cite[Lemma~2.4]{GLS08}.
Thus, for any such $\phi$, we will denote by $\homog{\phi}{d}$ 
the unique homogenous function in $\irrep{d\omega}$ that maps to $\phi$ under $\dehomog$.

\begin{definition}\label{def:homcc}
For any $M\in \CM(A)$, its \emph{homogeneous cluster character} is
\[
 \homcc{M} = \homog{\cluschar{\trun M}}{\rk M}.
\]
\end{definition}
Thus $\deg\homcc{M}=\rk M$ and 
the definition is sound because, by Theorem~\ref{thm:main} and \cite[\S10.1]{GLS08}, we know
$\rk M \geq \dim\soc \trun M = \deg \cluschar{\trun M}$.

\begin{lemma}
The map $M\mapsto \homcc{M}$ is a cluster character, in the sense that
\[
  (a) \qquad \homcc{M_1\oplus M_2}=\homcc{M_1}\homcc{M_2}
\]
and, if $\dim\Ext^1(X,Y)=1$ (and so also $\dim\Ext^1(Y,X)=1$) and 
$Y\to A\to X$ and $X\to B\to Y$ are the corresponding non-split short exact sequences,
then
 \[
  (b) \qquad \homcc{X}\homcc{Y}=\homcc{A} + \homcc{B}.
\]  
Furthermore, $\homcc{M}$ depends on $M$ only up to isomorphism.
\end{lemma}

\begin{proof}
Observe first that the homogeneous functions on both sides of the equalities (a) and (b) 
have the same degree, because $\rk M$ is additive on direct sums and short exact sequences.
It is then sufficient to note that $\trun\colon \CM(A) \to \Sub Q_k$ is an exact functor
and the equalities hold when $\homcc{M}$ is replaced by $\cluschar{\trun M}$,
that is, $N\mapsto \cluschar{N}$ is a cluster character,
by \cite[Lemma~7.3]{GLS05} and \cite[Theorem~2]{GLS07b}.
This also yields the final claim,
because $\cluschar{\trun M}$ depends on $\trun M$ only up to isomorphism.
\end{proof}

We may first calculate that $\homcc{P_n}=\homog{\cluschar{0}}{1}= \Pluck{\cycint{n}}$,
as $\cluschar{0}=1$.
More generally, if $r$ is the number of summands of $M$ isomorphic to $P_n$, then
\begin{equation}\label{eq:psiphitil}
 \homcc{M} = \cluschartil{\trun M} (\Pluck{\cycint{n}})^r 
\end{equation}
where $\cluschartil{N}= \homog{\cluschar{N}}{\deg \cluschar{N}}$ 
is the minimally homogenised cluster character of $N$,
as in \cite[\S10.1]{GLS08}.
In particular, if $M$ is the minimal lift of $N$, in the sense of Theorem~\ref{thm:main},
then $\homcc{M} = \cluschartil{N}$.

\New{
Following Remark~\ref{rem:class}, for the submodules $N_I\leq Q_k$
we know that, except in the case $I=\cycint{n}$, the minimal lift of $N_I$ is $L_I$
and $\cluschartil{N_I}=\Pluck{I}$. 
On the other hand, $L_{\cycint{n}}=P_n$ and so,
for all Pl\"ucker labels $I$ without exception, we have
\begin{equation}\label{eq:psi}
\homcc{L_I} = \Pluck{I}.
\end{equation}
}

\goodbreak
\begin{proposition}\label{prop:GLweight}
The $\GL_n(\k)$ weight of $\homcc{M}$ is $[M]\in\Grot(A)\isom \ZZ^n(k)$. 
\end{proposition}

\begin{proof}
We know from \cite[Lemma~5.4]{GLS07a} that, for any $N \leq Q_{d\omega}$, 
the $\SL_n(\k)$ weight of $\homog{\cluschar{N}}{d}$ is $d\omega + [N]$,
where the class $[N]\in\Grot(B)$,
i.e. the dimension vector of $N$, 
is interpreted in the root lattice of $\SL_n(\k)$ 
following \eqref{eq:ident}; that is, $[N] = \sum_i (\dim N_i) \alpha_i$,
where the $\alpha_i$ are the negative simple roots, as in \eqref{eq:alpha}.
Hence the $\GL_n(\k)$ weight of $\homog{\cluschar{N}}{d}$
is also $d\omega + [N]$, provided we interpret $\omega$ as
the appropriate $\GL_n(\k)$ weight,
i.e. set $\omega = \beta_{\cycint{n}} \in \ZZ^n(k)$,
as in \eqref{eq:beta}.

By \eqref{eq:Grot}, we have $[M] = [\trun M] + \rk(M)[P_n]$, and so the claim follows
because $[P_n]= \beta_{\cycint{n}}$ under the identification $\Grot(A)\isom \ZZ^n(k)$,
as in \eqref{eq:ident}.
\end{proof}

\begin{proposition}\label{prop:NtoM}
For any submodule $N\leq Q_k^{\oplus d}$, there is (up to isomorphism) a unique $M\in\CM(A)$ such that
$\rk M =d$ and $\trun M\isom N$.
Conversely, if $M\in\CM(A)$ has $\rk M =d$, then $\trun M$ is isomorphic to a submodule of $Q_k^{\oplus d}$.
Thus 
\begin{equation}\label{eq:psispan}
  \irrep{d\omega} = \langle \homcc{M} : M \in\CM(A),\; \rk M=d \rangle.
\end{equation}
\end{proposition}

\begin{proof}
The fact that $N\leq Q_k^{\oplus d}$ implies that $\dim\soc N \geq d$.
Hence we can (and indeed must) take $M=(P_n)^r\oplus \Ntil$, where $\Ntil\in\CM(A)$ is the minimal lift of $N$ provided by Theorem~\ref{thm:main}
and $r= d - \dim\soc N$. 
For the converse, if $r$ is the number of summands of $M$ isomorphic to $P_n$,
then $\dim\soc \trun M = d-r$ and so $\trun M\leq Q_k^{\oplus (d-r)}\leq Q_k^{\oplus d}$. 

By \eqref{eq:phispan} and \eqref{eq:projection} we have 
\[
  \irrep{d\omega} = \langle \homog{\cluschar{N}}{d} : N \leq Q_k^{\oplus d} \rangle.
\]
The second claim \eqref{eq:psispan} then follows immediately from the first and Definition~\ref{def:homcc}.
\end{proof}

\New{Recall that a cluster tilting object is \emph{reachable} if it can be obtained from an `initial' cluster tilting object,
corresponding to an initial cluster, by mutation.
A rigid indecomposable object is reachable if it is a summand of a reachable cluster tilting object.
For $\GCA{k}{n}$, the standard initial clusters consist of Pl\"ucker coordinates \cite[\S4]{Sc06}
and hence the initial cluster tilting object has rank one summands.
}

\begin{theorem}\label{thm:clusvar}
The map $M\mapsto \homcc{M}$ induces a one-to-one correspondence between
the (isomorphism classes of) \New{reachable} rigid indecomposable modules of rank $d$ in $\CM(A)$
and the cluster variables of degree $d$ in $\GCA{k}{n}$.
\end{theorem}

\begin{proof}
If $M\in \CM(A)$ is indecomposable, 
then either $M=\widetilde{\trun M}$ is a minimal lift of an indecomposable
in $\Sub Q_k$ and $\homcc{M}=\cluschartil{\trun M}$,
or $M=P_n$ and $\homcc{M}=\Pluck{\cycint{n}}$.
The result then follows because we know first (see Remarks~\ref{rem:main} \& \ref{rem:CTO})
that $M\mapsto \trun M$ gives a bijection between the \New{reachable} rigid indecomposables
in $\CM(A)$ except $P_n$ and the \New{reachable} rigid indecomposables in $\Sub Q_k$
and second (see \cite[\S10.2]{GLS08}) that $N\mapsto \cluschartil{N}$
gives a bijection between the \New{reachable} rigid indecomposables in $\Sub Q_k$ and the cluster variables
in $\GCA{k}{n}$ except $\Pluck{\cycint{n}}$.
\end{proof}

\begin{remark}\label{rem:clustercorresp}
Geiss-Leclerc-Schr\"oer \cite[Thm.~9.3]{GLS08} 
also use their cluster character map $N\mapsto \cluschar{N}$ to induce
a bijection between \New{reachable} cluster tilting objects in $\Sub Q_k$ and clusters in $\k[\N]$,
or \cite[Thm.~10.2]{GLS08} with clusters in $\GCA{k}{n}$ by homogenisation
and including $\Pluck{\cycint{n}}$.

We may now formulate this more simply as saying that the bijection $M\mapsto \homcc{M}$ induces a bijection between \New{reachable} cluster tilting objects in $\CM(A)$ and clusters in $\GCA{k}{n}$.
In this way, $\CM(A)$ is a more natural categorification of the cluster structure on $\GCA{k}{n}$.
\end{remark}


\affiliationone{B.T. Jensen, NTNU Norwegian Univ. of Science and Technology, 
Teknologvn. 22, 2815 Gj\o vik, Norway
\email{bernt.jensen@ntnu.no}}
\affiliationtwo{A. King \& X.Su, Mathematical Sciences, Univ. of Bath, 
Bath BA2 7AY, U.K.
\email{A.D.King@bath.ac.uk\\ X.Su2@bath.ac.uk}}

\end{document}